\pdfoutput=1
\RequirePackage{ifpdf}
\ifpdf % We are running pdfTeX in pdf mode
\documentclass[pdftex]{sigma}
\else
\documentclass{sigma}
\fi

\usepackage{mathrsfs}

\numberwithin{equation}{section}
\newtheorem{Theorem}{Theorem}[section]
\newtheorem{Corollary}[Theorem]{Corollary}
\newtheorem{Lemma}[Theorem]{Lemma}
\newtheorem{Proposition}[Theorem]{Proposition}
{\theoremstyle{definition}
\newtheorem{Definition}[Theorem]{Definition}

\newtheorem{Remark}[Theorem]{Remark}
\newtheorem{Notation}[Theorem]{Notation}
}

\newcommand{\nc}{\R}
\newcommand{\g}{\mathfrak{g}}
\newcommand{\h}{\mathfrak{h}}
\newcommand{\kk}{\mathfrak{k}}
\newcommand{\ttt}{\mathfrak{t}}
\newcommand{\p}{\mathfrak{p}}
\newcommand{\n}{\mathfrak{n}}
\newcommand{\llll}{\mathfrak{l}}
\newcommand{\uu}{\mathfrak{u}}
\newcommand{\bb}{\mathfrak{b}}
\newcommand{\q}{\mathfrak{q}}
\newcommand{\su}{\mathfrak{su}}
\newcommand{\ssl}{\mathfrak{sl}}

\newcommand{\G}{\mathbb{G}}

\newcommand{\C}{\mathbb{C}}
\newcommand{\R}{\mathbb{R}}
\newcommand{\Z}{\mathbb{Z}}
\newcommand{\N}{\mathbb{N}}
\newcommand{\Hh}{\mathbb{H}}
\newcommand{\X}{\mathbb{X}}

\newcommand{\ZZ}{\mathscr{Z}}
\newcommand{\ad}{\mathrm{ad}}
\newcommand{\Hsp}{\mathscr{H}}
\newcommand{\qn}[2]{\lbrack #1 \rbrack_{#2}}

\newcommand{\Tr}{\mathrm{Tr}}

\newcommand{\Ad}{\mathrm{Ad}}
\newcommand{\du}[1]{#1^{\vee}}

\newcommand{\mmod}{\,\mathrm{mod}\,}
\DeclareMathOperator{\id}{id}
\DeclareMathOperator{\ext}{\mathrm{ext}}

\DeclareMathOperator{\fin}{\mathrm{fin}}
\DeclareMathOperator{\Pol}{\mathrm{Pol}}

\DeclareMathOperator{\reg}{\mathrm{reg}}
\DeclareMathOperator{\sgn}{\mathrm{sgn}}

\DeclareMathOperator{\Ind}{\mathrm{Ind}}

\DeclareMathOperator{\Char}{\mathrm{Char}}
\DeclareMathOperator{\triv}{\mathrm{triv}}
\DeclareMathOperator{\sss}{\mathrm{ss}}

\begin{document}

\allowdisplaybreaks

\renewcommand{\thefootnote}{$\star$}

\renewcommand{\PaperNumber}{081}

\FirstPageHeading

\ShortArticleName{Representations of Interpolating QUE Algebras}

\ArticleName{Representation Theory of Quantized Enveloping\\
Algebras with Interpolating Real Structure\footnote{This paper is a~contribution to the
Special Issue on Noncommutative Geometry and Quantum Groups in honor of Marc A.~Rieffel.
The full collection is available at \href{http://www.emis.de/journals/SIGMA/Rieffel.html}
{http://www.emis.de/journals/SIGMA/Rieffel.html}}}

\Author{Kenny DE COMMER}

\AuthorNameForHeading{K.~De Commer}

\Address{Department of Mathematics, University of Cergy-Pontoise,\\
UMR CNRS 8088, F-95000 Cergy-Pontoise, France}
\Email{\href{mailto:Kenny.De-Commer@u-cergy.fr}{Kenny.De-Commer@u-cergy.fr}}
\URLaddress{\url{http://kdecommer.u-cergy.fr}}

\ArticleDates{Received August 18, 2013, in final form December 18, 2013; Published online December 24, 2013}

\Abstract{Let $\g$ be a~compact simple Lie algebra.
We modify the quantized enveloping $^*$-algebra associated to $\g$ by a~real-valued character on the
positive part of the root lattice.
We study the ensuing Verma module theory, and the associated quotients of these modified quantized
enveloping $^*$-algebras.
Restricting to the locally finite part by means of a~natural adjoint action, we obtain in particular
examples of quantum homogeneous spaces in the operator algebraic setting.}


\Classification{17B37; 20G42; 46L65}

\renewcommand{\thefootnote}{\arabic{footnote}}
\setcounter{footnote}{0}

\pdfbookmark[1]{Introduction}{intro} \section*{Introduction}

This paper reports on preliminary work related to the quantization of non-compact semi-simple Lie groups.
The main idea behind such a~quantization is based on the reflection technique developed in~\cite{DeC2}
and~\cite{Eno1} (see also~\cite{DeC3} and~\cite{DeC4} for concrete, small-dimensional examples relevant to
the topic of this paper).
Briefly, this technique works as follows.
Let $\G$ be a~compact quantum group acting on a~compact quantum homogeneous space~$\X$.
Assume that the von Neumann algebra $\mathscr{L}^{\infty}(\X)$ associated to~$\X$ is a~type $I$ factor.
Then the action of $\G$ on $\mathscr{L}^{\infty}(\X)$ can be interpreted as a~projective representation
of~$\G$, and one can deform $\G$ with the `obstruction' associated to this projective representation to
form a~new \emph{locally} compact quantum group~$\Hh$.
More generally, if $\mathscr{L}^{\infty}(\X)$ is only a~finite direct sum of type $I$-factors, one can
construct $\Hh$ as a~locally compact quantum group\emph{oid} (of a~particularly simple type).
Our idea is to fit the quantizations of non-compact semi-simple Lie groups into this framework, obtaining
them as a~reflection of the quantization of their compact companion.
For this, one needs the proper quantum homogeneous spaces to feed the machinery with.

It is natural to expect the needed quantum homogeneous space to be a~quantization of a~compact symmetric
space associated to the non-compact semi-simple Lie group.
By now, there is much known on the quantization of symmetric spaces (see~\cite{Let1,Let2} and references
therein, and~\cite{Vak1} for the non-compact situation), but these results are mostly of an algebraic
nature, and not much seems known about corresponding operator algebraic constructions except for special
cases.
In fact, in light of the motivational material presented in Appendix~\ref{SecLieAlg}, we will instead of
symmetric spaces use certain quantizations of (co)adjoint orbits, following the approach
of~\cite{Don2,Kar1,Don1}.
Here, one rather constructs quantum homogeneous spaces as subquotients of (quantized) universal enveloping
algebras in certain highest weight representations.
We will build on this approach by combining it with real structures and the contraction technique.

Our main result, Theorem~\ref{TheoFI}, will consist in showing that the compact quantum homogeneous spaces
that we build do indeed consist of finite direct sums of type~I factors.
This will give a~theoretical underpinning and motivation for the claim that the above mentioned
quantizations of non-compact semi-simple groups can indeed be constructed using the reflection technique.
Our results are however quite incomplete as of yet, as
\begin{itemize}\itemsep=0pt
\item[$\bullet$] in the non-contracted case, we can only treat concretely the case of
Hermitian symmetric spaces, \item[$\bullet$] a~more detailed analysis of the resulting quantum homogeneous
spaces is missing, \item[$\bullet$] the relation to known quantum homogeneous spaces is not elucidated,
\item[$\bullet$] no precise connection with deformation quantization is provided, \item[$\bullet$] the
relation with the approach of Korogodsky~\cite{Kor1} towards the quantization of non-compact Lie groups
remains to be clarified.
\end{itemize}

We hope to come back to the above points in future work.

The structure of this paper is as follows.
In Section~\ref{section1}, we introduce the `modified'
quantized universal enveloping algebras we will be studying,
and state their main properties in analogy with the ordinary quantized universal enveloping algebras.
In Section~\ref{section2}, we introduce a~theory of Verma modules, and study the associated
unitarization problem.
In Section~\ref{section3}, we study subquotients of our generalized quantized universal enveloping
algebras, and show how they give rise to $C^*$-algebraic quantum homogeneous spaces whose associated von
Neumann algebras are direct sums of type~$I$ factors.
In Section~\ref{section4}, we briefly discuss a~case where the associated von Neumann algebra is
simply a~type $I$-factor itself.

In the appendices, we give some further comments on the structures appearing in this paper.
In Appendix~\ref{SecCog}, we recall the notion of cogroupoids~\cite{Bic2} which is very convenient
for our purposes.
In Appendix~\ref{SecLieAlg}, we discuss the Lie algebras which are implicitly behind the
constructions in the main part of the paper.

\section{Two-parameter deformations of quantized\\
enveloping algebras}\label{section1}

Let $\g$ be a~complex simple Lie algebra of rank $l$, with fixed Cartan subalgebra $\h$ and Cartan
decomposition $\g = \n^-\oplus \h \oplus \n^+$.
Let $\Delta \subseteq \h^*$ be the associated finite root system, $\Delta^+$ the set of positive roots,
and $\Phi^+ = \{\alpha_r\,|\, r\in I\}$ the set of simple positive roots.
We identify $I$ and $\Phi^+$ with the set $\{1,\ldots,l\}$ whenever convenient.
Let $\h_{\R}^* \subseteq \h^*$ be the real linear span of the roots, and let $(\;,\,)$ be an inner product on
$\h_{\R}^*$ for which $A = (a_{rs})_{r,s\in I} = \left( (\du{\alpha}_r,\alpha_s)\right)_{r,s\in I}$ is the
Cartan matrix of $\g$, where $\du{\alpha} = \frac{2}{(\alpha,\alpha)}\alpha$ for $\alpha\in \Delta$.
Let $\h_{\R}\subseteq \h$ be the real linear span of the coroots $h_{\alpha}$, where $\beta(h_{\alpha}) =
(\beta,\du{\alpha})$ for $\alpha,\beta\in \Delta$.

We further use the following notation.
We write $\{\omega_r\,|\, r\in I\}$ for the fundamental weights in $\h^*_{\R}$, so $(\omega_r,\du{\alpha}_s)
= \delta_{rs}$.
The $\Z-$lattice spanned by $\{\omega_r\}$ is denoted $P\subseteq \h^*_{\R}$, and $P^+$ denotes elements
expressed as positive linear combinations of this basis.
Similarly, the root lattice spanned by the $\alpha_r$ is denoted $Q$, and its positive span by $Q^+$.
We write $\Char_K(F)$ for the monoid of monoid homomorphisms from a~commutative (additive) monoid $(F,+)$
to a~commutative (multiplicative) monoid $(K,\cdot)$.
For $\varepsilon\in \Char_K(Q^+)$ we will abbreviate $\varepsilon_{\alpha_r} = \varepsilon_r$.
The unit element of $\Char_K(Q^+)$ will be denoted $+$, while the element $\varepsilon$ such that
$\varepsilon_r =0$ for all $r$ will be denoted 0.

We use the following notation for $q$-numbers, where $0<q<1$ is fixed for the rest of the paper:
\begin{itemize}\itemsep=0pt
\item[$\bullet$] for $r\in \Phi^+$, $q_r = q^{\frac{(\alpha_r,\alpha_r)}{2}}$,
\item[$\bullet$] for
$n\geq0$, $\qn{n}{r} = \frac{q_r^{n}-q_r^{-n}}{q_r-q_r^{-1}}$,
\item[$\bullet$] for $n\geq 0$, $\qn{n}{r}!
= \prod\limits_{k=1}^n \qn{k}{r}$,
\item[$\bullet$] for $m\geq n\geq 0$, $\left[\begin{matrix}m \\ n\end{matrix}\right]_r =
\frac{\qn{m}{r}!}{\qn{n}{r}!\qn{m-n}{r}!}$,

\item[$\bullet$] for $\alpha\in \h^*$, we define $q^{\alpha}
\in \Char_{\C}(P)$ by $q^{\alpha}(\omega) = q^{(\alpha,\omega)}$ (where $(\;,\,)$ has been $\C$-linearly
extended to $\h$).
\end{itemize}

We will only work with unital algebras defined over $\C$, and correspondingly all tensor products are
algebraic tensor products over $\C$.
By a~\emph{$^*$-algebra} $A$ will be meant an algebra $A$ endowed with an anti-linear, anti-multiplicative
involution $^*$: $A\rightarrow A$.
We further assume that the reader is familiar with the theory of Hopf algebras.
A \emph{Hopf $^*$-algebra} $(H,\Delta)$ is a~Hopf algebra whose underlying algebra $H$ is a~$^*$-algebra,
and whose comultiplication $\Delta$ is a~$^*$-homomorphism.
This implies that the counit is a~$^*$-homomorphism, and that the antipode $S$ is invertible with
$S^{-1}(h) = S(h^*)^*$ for all $h\in H$.
\begin{Definition}
For $\varepsilon,\eta\in \Char_{\R}(Q^+)$, we define $U_q(\g;\varepsilon,\eta)$ as the universal unital
$^*$-algebra generated by couples of elements $X_r^{\pm}$, $r\in I$, as well as elements $K_{\omega}$,
$\omega\in P$, such that for all $r,s \in I$ and $\omega,\chi\in P$, we have $K_{\omega}$ self-adjoint,
$(X_r^+)^*=X_r^-$ and
\begin{enumerate}\itemsep=0pt
\item[(K)$^q$] $K_{\omega}$ is invertible and $K_{\chi}K_{\omega}^{-1} = K_{\chi-\omega}$,
\item[(T)$^q$]
$K_{\omega}X_{r}^{\pm}K_{\omega}^{-1} = q^{\frac{\pm (\omega,\alpha_r)}{2}} X_{r}^{\pm}$,
\item[(S)$^q$]
$\sum\limits_{k=0}^{1-a_{rs}} (-1)^k \left[\begin{matrix}1-a_{rs} \\ k \end{matrix}\right]_r  (X_r^{\pm})^kX_s^{\pm}(X_r^{\pm})^{1-a_{rs}-k} =
0$ for $r\neq s$,

\item[(C)$^q_{\varepsilon,\eta}$] $\lbrack X_r^+,X_s^-\rbrack = \delta_{rs} \frac{\varepsilon_r
K_{\alpha_r}^2 - \eta_r K_{\alpha_r}^{-2}}{q_r-q_r^{-1}}$.
\end{enumerate}

When $\varepsilon=\eta=+$ (i.e.\
{}$\varepsilon_r=\eta_r=1$ for all $r$), we will denote the underlying algebra as $U_q(\g)$.
This is a~slight variation, obtained by considering $\frac{1}{2}P$ instead of $P$, of the simply-connected
version of the ordinary quantized universal enveloping algebra of $\g$~\cite[Remark 9.1.3]{Cha1}.
\end{Definition}

The $U_q(\g;\varepsilon,\varepsilon)$ are quantizations of the $^*$-algebras $U(\g_{\varepsilon})$,
introduced in Appendix~\ref{SecLieAlg}.
Note also that the algebras $U_q(\g;0,+)$ are well-known to algebraists, see~\cite[Section~3]{Kas1} and~\cite{Jos2}.

\looseness=1
Remark that the algebras $U_q(\g;\varepsilon,\eta)$ with $\prod_{r}\varepsilon_r\eta_r\neq 0$ are
mutually isomorphic as unital algebras.
Indeed, as $A$ is invertible, we can choose $b\in \Char_{\C}(P)$ such that $b_{\alpha_r}^4 =
\eta_r/\varepsilon_r$ for all $r\in \Phi^+$.
We can also choose $a\in \Char_{\C}(Q^+)$ such that $a_r$ is a~square root of $b_{\alpha_r}^2\varepsilon_r
= \frac{\eta_r}{b_{\alpha_r}^2}$.
Then
\begin{gather}
\label{Iso}
\phi\colon  \ U_q(\g;\varepsilon,\eta)\rightarrow U_q(\g),
\qquad
\begin{cases}
X_r^+ \mapsto a_rX_r^+,
\\
X_r^- \mapsto a_rX_r^-,
\\
K_{\omega} \mapsto b_{\omega}K_{\omega}
\end{cases}
\end{gather}
is a~unital isomorphism.
However, unless $\sgn(\varepsilon_r\eta_r)=+$ for all $r$, in which case we can choose $b\in
\Char_{\R^+}(P)$ and $a\in \Char_{\R^+}(Q^+)$, this rescaling will not respect the $^*$-structure.

We list some properties of the $U_q(\g;\varepsilon,\eta)$.
The proofs do not differ from those for the well-known case $U_q(\g)$, cf.~\cite{Ros1} and~\cite[Section~4]{Jos1}.

{\samepage \begin{Proposition}\label{PropEasy}\qquad
\begin{enumerate}\itemsep=0pt
\item[$1.$] Let $U_q(\n^{\pm})$ be the unital subalgebra generated by the $X_r^{\pm}$ inside
$U_q(\g;\varepsilon,\eta)$.
Then $U_q(\n^{\pm})$ is the universal algebra generated by elements $X_r^{\pm}$ satisfying the relations~$({\rm S})^q$, and in particular does not depend on $\varepsilon$ or $\eta$.
\item[$2.$] Similarly, let $U_q(\bb^{\pm})$ be the subalgebra generated by $U_q(\n^{\pm})$ and $U(\h)$, where
$U(\h)$ is the algebra generated by all $K_{\omega}$.
Then $U_q(\bb^{\pm})$ is universal with respect to the relations $({\rm K})^q$, $({\rm T})^q$ and $({\rm S})^q$.
\item[$3.$] $($Triangular decomposition$)$ The multiplication map gives an isomorphism
\begin{gather}
\label{EqTriang}
U_q\big(\n^+\big)\otimes U(\h)\otimes U_q(\n^-)\rightarrow U_q(\g;\varepsilon,\eta).
\end{gather}
\end{enumerate}
\end{Proposition}}

In particular, the $U_q(\g;\varepsilon,\eta)$ are non-trivial.
Note that the above proposition allows one to identify $U_q(\g;\varepsilon,\eta)$ and $U_q(\g)$ as
\emph{vector spaces},
\begin{gather*}
i_{\varepsilon,\eta}\colon  \ U_q(\g;\varepsilon,\eta)\rightarrow U_q(\g),
\qquad
x_+x_0x_-\mapsto x_+x_0x_-,
\qquad
x_{\pm}\in U_q\big(\n^{\pm}\big),
\qquad
x_0\in U(\h).
\end{gather*}
Hence $U_q(\g;\varepsilon,\eta)$ may also be viewed as $U_q(\g)$ with a~deformed product
$m_{\varepsilon,\eta}$.
Indeed, arguing as in~\cite[Section~4]{Kass1}, one can show that
\begin{gather*}
m_{\varepsilon,\eta}(x\otimes y)=\omega_{\varepsilon}(x_{(1)},y_{(1)})x_{(2)}y_{(2)}\omega_{\eta}(x_{(3)}
,y_{(3)})
\end{gather*}
for certain cocycles $\omega_{\varepsilon}$ on $U_q(\g)$.

Also the following proposition is immediate.
\begin{Proposition}
For each $\varepsilon,\mu,\eta \in \Char_{\R}(Q^+)$, there exists a~unique unital $^*$-homomorphism
\begin{gather*}
\Delta_{\varepsilon,\eta}^{\mu}
\colon \ U_q(\g;\varepsilon,\eta)\rightarrow U_q(\g;\varepsilon,\mu)\otimes U_q(\g;\mu,\eta)
\end{gather*}
such that $\Delta_{\varepsilon,\eta}^{\mu}(K_{\omega}) = K_{\omega}\otimes K_{\omega}$ for all $\omega\in
P$ and
\begin{gather*}
\Delta_{\varepsilon,\eta}^{\mu}\big(X_r^{\pm}\big)=X_r^{\pm}\otimes K_{\alpha_r}+K_{\alpha_r}^{-1}\otimes X_r^{\pm},
\qquad
\forall \, r\in\Phi^+.
\end{gather*}
\end{Proposition}
\begin{proof}
From Proposition~\ref{PropEasy} and the case $\varepsilon=\eta =+$, we know that
$\Delta_{\varepsilon,\eta}^{\mu}$ respects the relations $({\rm K})^q$, $({\rm T})^q$ and $({\rm S})^q$.
The compatibility with the relations $({\rm C})^q_{\varepsilon,\eta}$ is verified directly.
\end{proof}

\begin{Lemma}
The collection $\big(\{U_q(\g;\varepsilon,\eta)\},\{\Delta_{\varepsilon,\eta}^{\mu}\}\big)$ forms a~connected
cogroupoid over the index set $\Char_{\R}(Q^+)\cong \R^l$ $($cf.\ Appendix~{\rm \ref{SecCog})}.
\end{Lemma}
\begin{proof}
One immediately checks the coassociativity condition on the generators.
One further can define uniquely a~unital $^*$-homomorphism
\begin{gather*}
\epsilon_{\varepsilon}\colon  \ U_q(\g;\varepsilon,\varepsilon)\rightarrow\C,
\qquad
\begin{cases}
X_r^{\pm} \mapsto0,
\\
K_{\omega} \mapsto1
\end{cases}
\end{gather*}
and unital anti-homomorphism
\begin{gather*}
S_{\varepsilon,\eta}\colon  \ U_q(\g;\varepsilon,\eta)\rightarrow U_q(\g;\eta,\varepsilon),
\qquad
\begin{cases}
X_r^{\pm} \mapsto-q^{\pm1}X_r^{\pm},
\\
K_{\omega} \mapsto K_{\omega}^{-1}.
\end{cases}
\end{gather*}
Again, one verifies on generators that these maps satisfy the counit and antipode condition on
generators, hence on all elements.
\end{proof}

We will denote by $\lhd$ the associated adjoint action of $U_q(\g;\eta,\eta)$ on
$U_q(\g;\varepsilon,\eta)$, cf.\
Appendix~\ref{SecCog}, Definition~\ref{DefRA}.

\section{Verma module theory}\label{section2}

We keep the notation of the previous section.
\begin{Definition}
For $\lambda\in \Char_{\C}(P)$, we denote by $\C_\lambda$ the one-dimensional left $U_q(\bb^+)$-module
associated to the character
\begin{gather*}
\chi_{\lambda}\colon \ U_q\big(\bb^+\big)\rightarrow\C,
\qquad
\begin{cases}
X_r^+ \mapsto 0,
\\
K_{\omega} \mapsto \lambda_{\omega}.
\end{cases}
\end{gather*}
We denote
\begin{gather*}
M_{\lambda}^{\varepsilon,\eta}=U_q(\g;\varepsilon,\eta)\underset{U_q(\bb^+)}{\otimes}\mathbb{C}_{\lambda}.
\end{gather*}
We denote by $V_{\lambda}^{\varepsilon,\eta}$ the simple quotient of $M_{\lambda}^{\varepsilon,\eta}$ by
its maximal non-trivial submodule.
Note that such a~maximal submodule exists by the triangular decomposition~\eqref{EqTriang}, as then any
non-trivial submodule is a~sum of weight spaces with weights distinct from $\lambda$.
We denote by $v_{\lambda}^{\varepsilon,\eta}$ the highest weight vector $1\otimes 1$ in either
$M_{\lambda}^{\varepsilon,\eta}$ or $V_{\lambda}^{\varepsilon,\eta}$.
\end{Definition}

For $\lambda\in \Char_{\R}(P)$, we can introduce on $V_{\lambda}^{\varepsilon,\eta}$ a~non-degenerate
Hermitian form $\langle\;,\,\rangle$ such that $\langle xv,w\rangle = \langle v,x^*w\rangle$ for all $x\in
U_q(\g;\varepsilon,\eta)$ and $v,w\in V_{\lambda}^{\varepsilon,\eta}$.
We will call a~form satisfying this property \emph{invariant}.
Such a~form is then unique up to a~scalar.
The construction of this form is the same as in the case
of $U_q(\g)$~\cite[Section~5]{Jos1},~\cite[Section 2.1.5]{Vak1}.
Namely, let $U_q(\n^{\pm})_+\subseteq U_q(\n^{\pm})$ be the kernel of the restriction of the counit on
$U_q(\bb^{\pm})$, and consider the orthogonal decomposition
\begin{gather}
\label{EqTau}
U_q(\g;\varepsilon,\eta)=U(\h)\oplus\big(U_q\big(\n^-\big)_+U_q(\g;\varepsilon,\eta)+U_q(\g;\varepsilon,\eta)U_q(\n^+)_+\big).
\end{gather}
Let $\tau$ denote the projection onto the first summand.
Then one first observes that one has a~Hermitian form on $M_{\lambda}^{\varepsilon,\eta}$ by defining
\begin{gather*}
\langle xv_{\lambda}^{\varepsilon,\eta},yv_{\lambda}^{\varepsilon,\eta}\rangle=\chi_{\lambda}(\tau(x^*y)),
\qquad
x,y\in U_q(\g).
\end{gather*}
This is clearly a~well-defined invariant form.
It necessarily descends to a~non-degenerate Hermitian form on $V_{\lambda}^{\varepsilon,\eta}$.

The goal is to find necessary and sufficient conditions for this Hermitian form to be
positive-definite, in which case the module is called \emph{unitarizable}.
This is in general a~hard problem.
In the following, we present some partial results, restricting to the case $\eta = +$.
We always assume that $\lambda$ is real-valued, unless otherwise mentioned, and that
$M_{\lambda}^{\varepsilon,\eta}$ and $V_{\lambda}^{\varepsilon,\eta}$ have been equipped with the above
canonical Hermitian form.

We first consider the case $\varepsilon = \eta = +$, for which the following result is well-known.
\begin{Lemma}
\label{LemPosComp}
Let $\lambda\in \Char_{\R^+}(P)$.
Then the following are equivalent.
\begin{enumerate}\itemsep=0pt
\item[$1.$]
%\label{It1}
The module $V_{\lambda}^{+,+}$ is unitarizable.
\item[$2.$]
%\label{It2}
For all $r\in I$, $\lambda_{\alpha_r}^4 \in q_r^{2\N}$.
\item[$3.$]
%\label{It3}
 $V_{\lambda}^{+,+}$ is finite-di\-men\-sio\-nal.
\end{enumerate}
\end{Lemma}
\begin{proof}
For the implications $(2)\Rightarrow(3)\Rightarrow(1)$, see e.g.\
\cite[Corollary~10.1.15, Proposition~10.1.21]{Cha1}.
Assume now that $V_{\lambda}^{+,+}$ is unitarizable.
Then $V_{\lambda_r}^{+,+}$ is a~unitarizable representation of $U_{q_r}(\su(2))$.
By a~simple computation using the commutation between the $X_r^{\pm}$ (cf.~\cite{DeC6}), we have that
\begin{gather*}
\langle(X_r^+)^k(X_r^-)^kv_{\lambda_r}^{+,+},v_{\lambda_r}^{+,+}\rangle=\prod_{l=1}^k\frac{\big(q_r^l-q_r^{-l}\big)
\big(q_r^{-l+1}\lambda_{\alpha_r}^2-q_r^{l-1}\lambda_{\alpha_r}^{-2}\big)}{\big(q_r-q_r^{-1}\big)^2}.
\end{gather*}
Hence, by unitarity and the fact that $0<q_r<1$, we find that $\prod\limits_{l=1}^k
\big(q_r^{l-1}\lambda_{\alpha_r}^{-2}-q_r^{-l+1}\lambda_{\alpha_r}^2\big)\geq 0$ for all $k\geq0$.
This is only possible if eventually one of the factors becomes zero, i.e.\
when $\lambda_{\alpha_r}^4\in q_r^{2\N}$.
\end{proof}

By a~limit argument, we now extend this result to the case $\varepsilon \in \Char_{\{0,1\}}(Q^+)$.
\begin{Proposition}
\label{PropPosDef0}
Suppose $\varepsilon \in \Char_{\{0,1\}}(Q^+)$, and let $\lambda\in \Char_{\R^+}(P)$.
Then $V_{\lambda}^{\varepsilon,+}$ is unitari\-zab\-le if and only if $\lambda_{\alpha_r}^4 \in q_r^{2\N}$ for
all $r\in I$ with $\varepsilon_{r}=1$.

\end{Proposition}
\begin{proof}
The `only if' part of the proposition is obvious, since the $K_{\alpha_r},X_r^{\pm}$ with $\varepsilon_r =
1$ generate a~quantized enveloping algebra $\widetilde{U}_q(\kk_{\sss})$ in its compact (non-simply
connected) form.

To show the opposite direction, consider first general $\varepsilon,\eta \in \Char_{\R}(Q^+)$.
For $\alpha\in Q^+$, denote by $U_q(\n^-)(\alpha)$ the finite-di\-men\-sio\-nal space of elements $X\in
U_q(\n^-)$ with $K_{\omega}XK_{\omega}^{-1} = q^{-\frac{1}{2}(\omega,\alpha)}X$.
Identifying $U_q(\n^-)$ with $M_{\lambda}^{\varepsilon,\eta}$ in the canonical way by means of $x \mapsto
xv_{\lambda}^{\varepsilon,\eta}$, we may interpret the Hermitian forms on the
$M_{\lambda}^{\varepsilon,\eta}$ as a~family of Hermitian forms
$\langle\,\cdot,\cdot\,\rangle^{\varepsilon,\eta}_{\lambda}$ on $U_q(\n^-)$.
It is easily seen that these vary continuously with $\varepsilon,\eta$ and $\lambda$ on each
$U_q(\n^-)(\alpha)$.
Assume now that $\eta = +$ and $\varepsilon_r>0$ for all $r$.
In this case $U_q(\g;\varepsilon,+)$ is isomorphic to $U_q(\g)$ as a~$^*$-algebra by a~rescaling of the
generators by positive numbers, cf.~\eqref{Iso}.
By Lemma~\ref{LemPosComp}, $\langle\,\cdot,\cdot\,\rangle^{+,+}_{\lambda}$ is positive semi-definite if
and only if $V_{\lambda}^{+,+}$ is finite-di\-men\-sio\-nal, and the the latter happens if and only if
$\lambda_{\alpha_r}^{4}\in q_r^{2\N}$ for all $r\in I$.
Hence we get by the above rescaling that $\langle\,\cdot,\cdot\,\rangle^{\varepsilon,+}_{\lambda}$ is
positive semi-definite if and only if $\lambda_{\alpha_r}^4\in q_{r}^{2\N}\varepsilon_{r}^{-1}$ for all
$r$.

Fix now a~subset $S$ of the simple positive roots, and put $\varepsilon_r = 1$ for $r\in S$.
Assume that $\lambda_{\alpha_r}^4\in q_r^{2\mathbb{N}}$ when $r\in S$.
For $r\notin S$ and $m_r\in \mathbb{N}$, define $\varepsilon_r = q_r^{2m_r} \lambda_{\alpha_r}^{-4}$.
From the above, we obtain that $\langle\,\cdot,\cdot\,\rangle^{\varepsilon,+}_{\lambda}$ is positive
semi-definite.
Taking the limit $m_r\rightarrow \infty$ for $r\notin S$, we deduce that
$\langle\,\cdot,\cdot\,\rangle^{\varepsilon,+}_{\lambda}$ is positive semi-definite for $\varepsilon_r=\delta_{r\in S}$.
\end{proof}

The above $V_{\lambda}^{\varepsilon,+}$ with pre-Hilbert space structure can also be presented more
concretely as generalized Verma modules (from which it will be clear that they are not finite dimensional
when $\varepsilon_r=0$ for some $r$).
We will need some preparations, obtained from modifying arguments in~\cite{Jos1}.
Note that to pass from the conventions in~\cite{Jos1} to ours, the $q$ in~\cite{Jos1} has to be replaced by
$q^{1/2}$, and the coproduct in~\cite{Jos1} is also opposite to ours.
However, as~\cite{Jos1} gives preference to the left adjoint action, while we work with the right adjoint
action, most of our formulas will in fact match.

We start with recalling a~basic fact.
\begin{Lemma}
\label{LemFinK}
Let $\omega\in P^+$, $\varepsilon\in \Char_{\R}(Q^+)$, and consider $K_{\omega}^{-4}\in
U_q(\g;\varepsilon,+)$.
Then $K_{\omega}^{-4}\lhd U_q(\g)$ is finite-di\-men\-sio\-nal, where $\lhd$ is the adjoint action $($cf.\
Definition~{\rm \ref{DefRA})}.
\end{Lemma}
\begin{proof}
One checks that the arguments of~\cite[Lemma 6.1, Lemma 6.2, Proposition 6.3, Proposition 6.5]{Jos1} are
still valid in our setting.
\end{proof}

Let $\varepsilon\in \Char_{\R}(Q^+)$.
Recall that the map $\tau$ denoted the projection onto the first summand in~\eqref{EqTau}.
Let $\tau_{Z,\varepsilon}$ be the restriction of $\tau$ to the center $\ZZ(U_q(\g;\varepsilon,+))$ of
$U_q(\g;\varepsilon,+)$.
This is an $\varepsilon$-modified Harish-Chandra map.
The usual reasoning shows that this is a~homomorphism into~$U(\h)$.

\begin{Lemma}
\label{LemHarish}
The map $\tau_{Z,\varepsilon}$ is a~bijection between $\ZZ(U_q(\g;\varepsilon,+))$ and the linear span of the set
$\left\{\sum\limits_{w\in W} \varepsilon_{\omega-w\omega}q^{(-2w\omega,\rho)}K_{-4w\omega}\,\Big|\,\omega \in P^+\right\}$,
where $\rho = \sum_r \omega_r$ is the sum of the fundamental weights and $W$ denotes the
Weyl group of $\Delta$ with its natural action on $\h^*$.

\end{Lemma}
\begin{proof}
Note first that $\omega-w\omega$ is inside $Q^+$ for all $\omega\in P^+$ and $w\in W$, so that
$\varepsilon_{\omega-w\omega}$ is well-defined.

Arguing as in~\cite[Section~8]{Jos1}, we can assign to any $\omega\in P^+$ an element $z_{\omega}$
in $\ZZ(U_q(\g;\varepsilon,+))$, uniquely determined up to a~non-zero scalar, such that
\begin{gather*}
K_{\omega}^{-4}\lhd U_q(\g)=\mathbb{C}z_{\omega}+\left(K_{\omega}^{-4}\lhd U_q(\g)_+\right),
\end{gather*}
where $U_q(\g)_+$ is the kernel of the counit and where $\lhd$ denotes the adjoint action.
More concretely, as $K_{\omega}^{-4} \lhd U_q(\g)$ is a~finite-di\-men\-sio\-nal right $U_q(\g)$-module, it is
semi-simple, and we have a~projection $E$ of $K_{\omega}^{-4} \lhd U_q(\g)$ onto the space of its invariant
elements.
We can then take $z_{\omega} = E\big(K_{\omega}^{-4}\big)$, and $z_{\omega}\in \ZZ(U_q(\g;\varepsilon,+))$ by
Lemma~\ref{LemCenter}.

Suppose now first that $\varepsilon_r\neq 0$ for all $r$, and choose $b\in \Char_{\C}(P)$ such that
$b_{\alpha_r}^4 = \varepsilon_r^{-1}$.
Consider
\begin{gather*}
\Psi_{\varepsilon}\colon \ U_q(\g;\varepsilon,+)\rightarrow U_q(\g),
\qquad
\begin{cases}
X_r^{\pm} \mapsto  b_{\alpha_r}^{-1}X_r^{\pm},
\\
K_{\beta} \mapsto b_{\beta}K_{\beta},
\qquad
\beta\in P.
\end{cases}
\end{gather*}
This is a~unital $\lhd$-equivariant algebra isomorphism, and $\tau_{Z,\varepsilon} =
\Psi_{\varepsilon}^{-1}\circ \tau_{Z,+}\circ \Psi_{\varepsilon}$.
Hence by $\lhd$-equivariance, we find from the computations in~\cite[Section~8]{Jos1} that, for
a~non-zero $\varepsilon$-independent scalar $c_{\omega}$,
\begin{gather*}
\tau_{Z,\varepsilon}(z_{\omega})=c_{\omega}\sum_{\nu\in P^+}\dim((V_{\omega})_{\nu})\left(\sum_{w\in W}
b_{\omega-w\nu}^{-4}q^{-2(w\nu,\rho)}K_{-4w\nu}\right),
\end{gather*}
where $(V_{\omega})_{\nu}$ denotes the weight space at $q^{\frac{1}{2}\nu}$ (i.e.\
the space of vectors on which the $K_{\omega}$ act as~$q^{\frac{1}{2}(\nu,\omega)}$) of the
finite-di\-men\-sio\-nal $U_q(\g)$-module $V_{\omega}$ with highest weight $q^{\frac{1}{2}\omega}$.
But clearly we then only have to sum over those $\nu$ with $\omega-\nu \in Q^+$, so that
$b_{\omega-w\nu}^{-4} = \varepsilon_{\omega-w\nu}$, and we can write
\begin{gather}
\label{EqCharMod}
\tau_{Z,\varepsilon}(z_{\omega})=c_{\omega}\underset{\omega-\nu\in Q^+}{\sum_{\nu\in P^+}}\dim((V_{\omega}
)_{\nu})\left(\sum_{w\in W}\varepsilon_{\omega-w\nu}q^{-2(w\nu,\rho)}K_{-4w\nu}\right).
\end{gather}

Recall now that $U_q(\g;\varepsilon,\eta)$ can be identified with $U_q(\g)$ as a~vector space by a~map
$i_{\varepsilon,\eta}$.
Let us denote by $\lhd_{\varepsilon,\eta}$ the image of $\lhd$ under this map.
Then for $x,y\in U_q(\g)$ fixed, it is easily seen from the triangular decomposition that the
$x\lhd_{\varepsilon,\eta}y$ live in a~fixed finite-di\-men\-sio\-nal subspace of~$U_q(\g)$ as the
$\varepsilon,\eta$ vary, and the resulting map $(\varepsilon,\eta) \mapsto x\lhd_{\varepsilon,\eta}y$ is
then continuous.
Furthermore, if~$V$ is a~finite-di\-men\-sio\-nal right $U_q(\g)$-module with space of fixed elements
$V_{\triv}$, we can find $p\in U_q(\g)$ such that for any $v\in V$, the element~$vp$ is the projection of~$v$ onto~$V_{\triv}$.
It follows from the previous paragraph and the above remarks that when $\varepsilon_r\neq 0$ for any~$r$,
we have $i_{\varepsilon,+}(z_{\omega}) = K_{\omega}^{-4}\lhd_{\varepsilon,+} p_{\omega}$ for some
\emph{fixed} $p_{\omega}\in U_q(\g)$.
By continuity, it then follows that~\eqref{EqCharMod} in fact holds for arbitrary $\varepsilon\in
\Char_{\R}(Q^+)$.

The conclusion of the argument now follows as in~\cite[Theorem 8.6]{Jos1}.
\end{proof}

Let now $S\subseteq \Phi^+$, and let $\varepsilon$ extend the characteristic function of $S$.
Let $U_q(\ttt_S)$ be the Hopf $^*$-subalgebra of $U_q(\g;\varepsilon,+)$ generated by the
$K_{\omega_r}^{\pm}$ with $1\leq r \leq l$ and $X_r^{\pm}$ with $r\in S$.
Let $U_q(\q^+_{S})$ be the Hopf subalgebra of $U_q(\g;\varepsilon,+)$ generated by $U_q(\ttt_S)$ and all
$X_r^+$ with $r\in I$.
It is easy to see that $U_q(\q^+_{S})$ can be isomorphically imbedded into $U_q(\g)$.
Let $V$ be a~finite-di\-men\-sio\-nal highest weight representation of $U_q(\ttt_{S})$ associated to
a~character in $\Char_{\R^+}(P)$.
Then we can extend this to a~representation of $U_q(\q^+_S)$ on $V$~\cite[Section 2.3.1]{Vak1}, and
hence we can form
\begin{gather*}
\Ind_{\varepsilon}(V)=U_q(\g;\varepsilon,+)\underset{U_q(\q^+_S)}{\otimes}V.
\end{gather*}
The following proposition complements Proposition~\ref{PropPosDef0}.
\begin{Proposition}
\label{PropPosDef}
Let $S \subseteq \Phi^+$, and let $\varepsilon$ restrict to the characteristic function of $S$.
Let $V$ be an irreducible highest weight representation of $U_q(\ttt_S)$ associated to a~character $\lambda
\in \Char_{\R^+}(P)$.
Then the $U_q(\g;\varepsilon,+)$-representation $\Ind_{\varepsilon}(V)$ is irreducible.
\end{Proposition}
\begin{proof}
By its universal property, $\Ind_{\varepsilon}(V)$ can be identified with a~quotient of
$M_{\lambda}^{\varepsilon,+}$.
We want to show that this quotient coincides with $V_{\lambda}^{\varepsilon,+}$.

Suppose that $v_{\lambda'}$ is a~highest weight vector inside $M_{\lambda}^{\varepsilon,+}$ at weight
$\lambda'$ different from $\lambda$.
By Lemma~\ref{LemHarish}, it follows that
\begin{gather*}%\label{EqHar}
\sum_{w\in W}\varepsilon_{\omega-w\omega}q^{(-2w\omega,\rho)}\lambda_{-4w\omega}'=\sum_{w\in W}
\varepsilon_{\omega-w\omega}q^{(-2w\omega,\rho)}\lambda_{-4w\omega}
\end{gather*}
for all $\omega\in P^+$.
Now if $w = s_{\alpha_{i_1}}\cdots s_{\alpha_{i_p}}$ in reduced form, with $s_{\alpha}$ the reflection
across the root~$\alpha$, we have
\begin{gather*}
\omega-w\omega=\sum_{t=1}^p s_{\alpha_{i_1}}\cdots s_{\alpha_{i_{t-1}}}(\omega-s_{\alpha_{i_t}}\omega),
\end{gather*}
where each term is positive.
It follows that we have
\begin{gather*}%\label{EqHar2}
\sum_{w\in W_S}q^{(-2w\omega,\rho)}\lambda_{-4w\omega}'=\sum_{w\in W_S}q^{(-2w\omega,\rho)}
\lambda_{-4w\omega}
\end{gather*}
for all \emph{strictly} dominant $\omega$, where $W_S$ is the Coxeter group generated by reflections
around simple roots $\alpha_s$ with $s\in S$.

Taking $\omega = \rho + \omega_r$ with $r\notin S$, we get
$(\lambda_{-4\omega_r}'-\lambda_{-4\omega_r})C=0$ with
\begin{gather*}
C=\sum_{w\in W_S}q^{(-2w\rho,\rho)}\lambda_{-4w\rho}>0.
\end{gather*}
Hence $\lambda_{\omega_r}' = \lambda_{\omega_r}$ for all $r\notin S$.
We deduce that $v_{\lambda'} \in U_q(\ttt_S)v_{\lambda}$, and so the image of $v_{\lambda'}$ in
$\Ind_{\varepsilon}(V)$ is zero.
This implies that $\Ind_{\varepsilon}(V) = V_{\lambda}^{\varepsilon,+}$.
\end{proof}

The case $\varepsilon \in \Char_{\{-1,0,1\}}(Q^+)$ is not so easy to treat in general.
In the following, we will restrict ourselves to the case where we have one $\varepsilon_t<0$ at a root satisfying a particular condition, while $\varepsilon_r\geq 0$ for $r\neq t$.
This will correspond precisely to the `Hermitian symmetric' case.
\begin{Theorem}%\label{TheoHerm}
Let $\varepsilon$ be such that there is a~unique simple root $\alpha_t$ with $\varepsilon_t < 0$, while
$\varepsilon_r \in\{0,1\}$ for $r\neq t$.
Assume moreover that $\alpha_t$ appears with multiplicity at most $1$ in each positive root.
Let $\lambda \in \Char_{\R^+}(P)$.
Then $V_{\lambda}^{\varepsilon,+}$ is unitarizable if and only if $\lambda_{\alpha_r}^4\in q_r^{2\N}$ for
all $r$ with $\varepsilon_r\neq 1$.
\end{Theorem}

One can indeed check by a case-by-case analysis, using for example the tables in~\cite[Appendix~C]{Kna2}, that the Vogan diagram associated to a sign pattern will determine a Hermitian symmetric space precisely when the above multiplicity condition is satisfied.
\begin{proof}
By a~same kind of limiting argument as in Proposition~\ref{PropPosDef0}, the general case can be deduced
from the case with $\varepsilon_r=1$ for $r\neq t$.

Suppose then that $\varepsilon_t<0$ and $\varepsilon_r = 1$ for $r\neq t$, where $\alpha_r$ appears with multiplicity at most $1$ in each positive root.
Write $S = \Phi^+\setminus \{t\}$.
We have the algebra automorphism $\phi\colon  U_q(\g;\varepsilon,+)\rightarrow U_q(\g)$ appearing in~\eqref{Iso}.
By means of this isomorphism, we obtain a~natural isomorphism $M_{\lambda}^{\varepsilon,+}\cong
M_{\gamma}^{+,+}$, where $\gamma \in \Char_{\C}(P)$ is such that $\gamma_{\alpha_r}^4 =
\varepsilon_r\lambda_{\alpha_r}^4$ for all $r$.
In particular, $\gamma_{\alpha_t}^2 \in \C\setminus \R$.
By the condition we assume on $\alpha_r$, we can apply~\cite[Proposition~5.13]{Jos3} and deduce that $\Ind_{\varepsilon}(V_{\lambda})$ is irreducible, where
$V_{\lambda}$ denotes the irreducible representation of $U_q(\q^+_S)$ at highest weight $\lambda$.
Hence the signatures of the Hermitian inner products on the $\Ind_{\varepsilon}(V_{\lambda})$ are constant
as $\varepsilon_t<0$ varies.
Indeed, these spaces can be identified canonically with a~fixed quotient of $U_q(\n^-)$,
see~\cite[Proposition~2.81]{Vak1}, and then the Hermitian inner products clearly form a~continuous family
as $\varepsilon$ varies.

From Proposition~\ref{PropPosDef}, we know that the Hermitian inner product on
$\Ind_{\varepsilon}(V_{\lambda})$ for $\varepsilon_t=0$ is positive definite.
It follows that the Hermitian inner product on a~weight space of $\Ind_{\varepsilon}(V_{\lambda})$ is
positive for $\varepsilon_t<0$ small.
As the signature is constant, this holds for all $\varepsilon_t<0$.
\end{proof}

\section{Quantized homogeneous spaces}\label{section3}

\begin{Definition}
For $\varepsilon,\eta\in\Char_{\R}(Q^+)$, we denote by $U_q(\g;\varepsilon,\eta)_{\fin}$ the space of
locally finite vectors in $U_q(\g;\varepsilon,\eta)$ with respect to the right adjoint action by
$U_q(\g;\eta,\eta)$ (cf.\ Definition~\ref{DefRA}),
\begin{gather*}
U_q(\g;\varepsilon,\eta)_{\fin}=\big\{x\in U_q(\g;\varepsilon,\eta)\, \big|\dim(x\lhd U_q(\g;\eta,\eta))<\infty\big\}.
\end{gather*}
\end{Definition}

It is easily seen that the space $U_q(\g;\varepsilon,\eta)_{\fin}$ is a~$^*$-subalgebra of
$U_q(g;\varepsilon,\eta)$ (cf.\
\cite[Corollary~2.3]{Jos1}), and in the following it will always be treated as a~right $U_q(\g)$-module by~$\lhd$.

A similar definition of ${}_{\fin}U_q(\g;\varepsilon,\eta)$ can be made with respect to the left
adjoint action of~$U_q(\g;\varepsilon,\varepsilon)$, and the two resulting algebras
$U_q(\g;\varepsilon,\eta)_{\fin}$ and ${}_{\fin}U_q(\g;\varepsilon,\eta)$ should in some sense be seen
as dual to each other.
For example, the $U_q(\g;\varepsilon,+)_{\fin}$ will lead to compact quantum homogeneous spaces, while
the ${}_{\fin}U_q(\g;\varepsilon,+)$ should lead to non-compact quantum homogeneous spaces such as
quantum bounded symmetric domains~\cite{Vak1}.
However, in this paper we will restrict ourselves to the compact case.

The $U_q(\g;\varepsilon,\eta)_{\fin}$ are sufficiently large, as the next proposition shows,
extending Lemma~\ref{LemFinK}.
\begin{Proposition}
\label{PropBig}
As a~right $U_q(\g)$-module, $U_q(\g;\varepsilon,+)_{\fin}$ is generated by the $K_{\omega}^{-4}$ with
$\omega \in P^+$.
The algebra generated by $U_q(\g;\varepsilon,+)_{\fin}$ and the $K_{\omega_r}^4$ equals the subalgebra of
$U_q(\g;\varepsilon,+)$ generated by the $K_{\omega_r}^{\pm 4}$ and the $K_{\alpha_r}X_r^{\pm}$.
\end{Proposition}
\begin{proof}
Again, the proof of~\cite[Theorem 6.4]{Jos1} can be directly modified.
\end{proof}

Note that for $\varepsilon_r\neq 0$ for all $r$, the above proposition follows more straightforwardly
from~\cite{Jos1} by a~rescaling argument.

The dependence of $U_q(\g;\varepsilon,\eta)_{\fin}$ on $\varepsilon$ and $\eta$ is weaker than for
$U_q(\g;\varepsilon,\eta)$ itself.
We consider a~special case in the following lemma.
Recall that $A$ denotes the Cartan matrix.
\begin{Lemma}%\label{LemChange}
Consider $\varepsilon,\eta\in \Char_{\R\setminus\{0\}}(Q^+)$, and write $\sgn(\varepsilon_r/\eta_r) =
(-1)^{\chi_r}$.
If $\chi$ is in the range of $A \mmod 2$, then $U_q(\g;\varepsilon,+)_{\fin} \cong
U_q(\g;\eta,+)_{\fin}$ as right $U_q(\g)$-module $^*$-algebras.
\end{Lemma}
\begin{proof}
Choose $b\in \Char_{\C}(P)$ such that $b_{\alpha_r}^4 = \eta_r/\varepsilon_r$.
Then
\begin{gather*}
\Psi_{\varepsilon,\eta}\colon  \ U_q(\g;\varepsilon,+)\rightarrow U_q(\g;\eta,+),
\qquad
\begin{cases}
X_r^{\pm} \mapsto b_{\alpha_r}^{-1}X_r^{\pm},
\\
K_{\omega} \mapsto b_{\omega}K_{\omega}
\end{cases}
\end{gather*}
is a~unital $\lhd$-equivariant algebra isomorphism.
Hence $\Psi_{\varepsilon,\eta}$ induces a~unital $\lhd$-equivariant algebra isomorphism
$\psi_{\varepsilon,\eta}\colon U_q(\g;\varepsilon,+)_{\fin}\rightarrow U_q(\g;\eta,+)_{\fin}$.

As the $K_{\omega}^{-4}$ with $\omega\in P^+$ generate $U_q(\g;\varepsilon,+)_{\fin}$ as a~module,
$\psi_{\varepsilon,\eta}$ will be $^*$-preserving if and only if $b_{\alpha_r}^4 \in \mathbb{R}$ for all
$r$.
This can be realized if we can find $c_r \in \{-1,1\}$ such that $\prod_{s} c_s^{a_{sr}} =
\sgn(\eta_r/\varepsilon_r)$, which is equivalent with the condition appearing in the statement of the lemma.
\end{proof}

In particular, we find for example that $U_q(\ssl(2m+1);\varepsilon,+)_{\fin}$ for $m\in \N_0$ is
independent of the choice of $\varepsilon\in \Char_{\R\setminus\{0\}}(Q^+)$.
On the other hand, $U_q(\ssl(2);\varepsilon,+)_{\fin}$ are mutually non-isomorphic as $^*$-algebras for
$\varepsilon\in \{-1,0,1\}$, see~\cite{DeC1}.
\begin{Definition}
Let $\varepsilon,\eta\in \Char_{\R}(Q^+)$, $\lambda\in\Char_{\R^+}(P)$, and let
$V_{\lambda}^{\varepsilon,\eta}$ be the irreducible highest weight module of $U_q(\g;\varepsilon,\eta)$ at
$\lambda$ with associated representation $\pi_{\lambda}^{\varepsilon,\eta}$.
We write
\begin{gather*}
B_{\lambda}(\g;\varepsilon,\eta)=\pi_{\lambda}^{\varepsilon,\eta}(U_q(\g;\varepsilon,\eta))
\end{gather*}
and
\begin{gather*}
B_{\lambda}^{\fin}(\g;\varepsilon,\eta)=\pi_{\lambda}^{\varepsilon,\eta}(U_q(\g;\varepsilon,\eta)_{\fin}).
\end{gather*}
\end{Definition}
\begin{Remark}%\label{RemKost}
The space $B_{\lambda}^{\fin}(\g;\varepsilon,\eta)$ \emph{is not} defined as the space
$B_{\lambda}(\g;\varepsilon,\eta)_{\fin}$ of locally finite $\lhd$-elements in
$B_{\lambda}(\g;\varepsilon,\eta)$, although conceivably they are the same in many cases.
In the case $q=1$, the equality of these two algebras goes by the name of the Kostant problem, cf.\
\cite[Remark 3]{Kar1}.
\end{Remark}
\begin{Notation}
We will use the following notation for particular elements in the $B_{\lambda}(\g;\varepsilon,+)$:
\begin{gather*}
Z_r=\pi_{\lambda}^{\varepsilon,+}\big(K_{\omega_r}^{-4}\big),
\qquad
X_r=q_r^{1/2}(q_r^{-1}-q_r)\pi_{\lambda}^{\varepsilon,+}\big(K_{\alpha_r}K_{\omega_r}^{-4}X^+_r\big),
\\
Y_r=X_r^*,
\qquad
W_r=\pi_{\lambda}^{\varepsilon,+}\big(K_{\alpha_r}^4K_{\omega_r}^{-8}\big),
\\
T_r=\big(q_r-q_r^{-1}\big)^2\pi_{\lambda}^{\varepsilon,+}\big(K_{\alpha_r}^2K_{\omega_r}^{-4}
X_r^+X_r^-\big)+\varepsilon_rq_r^{-1}\pi_{\lambda}^{\varepsilon,+}\big(K_{\alpha_r}^4K_{\omega_r}^{-4}
\big)+q_r\pi_{\lambda}^{\varepsilon,+}\big(K_{\omega_r}^{-4}\big)
\\
\phantom{T_r}{}
=\big(q_r-q_r^{-1}\big)^2\pi_{\lambda}^{\varepsilon,+}\big(K_{\alpha_r}^2K_{\omega_r}^{-4}
X_r^-X_r^+\big)+\varepsilon_rq_r\pi_{\lambda}^{\varepsilon,+}\big(K_{\alpha_r}^4K_{\omega_r}^{-4}\big)+q_r^{-1}
\pi_{\lambda}^{\varepsilon,+}\big(K_{\omega_r}^{-4}\big).
\end{gather*}
\end{Notation}

The following commutation relations will be needed later on.
\begin{Lemma}
\label{LemSubAlg}
The elements $W_r$ and $T_r$ commute with $X_r$, $Y_r$, $Z_r$, $T_r$ and $W_r$.
Moreover,
\begin{gather*}
X_rZ_r=q_r^2Z_rX_r,\qquad Y_rZ_r=q_r^{-2}Z_rY_r
\end{gather*}
and
\begin{gather*}
X_r Y_r=-\varepsilon_r W_r+q_r T_r Z_r-q_r^2Z_r^2,
\qquad
Y_r X_r=-\varepsilon_r W_r+q_r^{-1}T_r Z_r-q_r^{-2}Z_r^2.
\end{gather*}
We further have that $T_r$ and $W_r$ are invariant under $\lhd X_r^{\pm}$ and $\lhd K_{\omega}$, while
\begin{gather*}
X_r\lhd X_r^+=0,\qquad Y_r\lhd X_r^+=-q_r^{1/2}\big(q_r^{-1}+q_r\big)Z_r+q_r^{1/2}T_r,
\\
X_r\lhd K_{\omega}=q^{-\frac{(\omega,\alpha_r)}{2}}X_r,\qquad Y_r\lhd K_{\omega}=q^{\frac{(\omega,\alpha_r)}{2}}
X_r,
\\
X_r\lhd X_r^-=q_r^{-1/2}\big(q_r^{-1}+q_r\big)Z_r-q_r^{-1/2}T_r, \qquad Y_r\lhd X_r^-=0
\end{gather*}
and
\begin{gather*}
Z_r\lhd X_{r}^+=q_r^{1/2}X_r,
\qquad
Z_r\lhd K_{\omega}=0,
\qquad
Z_r\lhd X_r^-=-q_r^{-1/2}Y_r.
\end{gather*}

Finally, all elements $X_r$, $Y_r$, $Z_r$, $T_r$, $W_r$ are inside $B_{\lambda}^{\fin}(\g;\varepsilon,+)$.
\end{Lemma}
\begin{proof}
All these assertions follow from straightforward computations.
As the $Z_r = \pi_{\lambda}^{\varepsilon,+}(K_{\omega_r}^{-4})$ and $W_r$ are in
$B_{\lambda}^{\fin}(\g;\varepsilon,+)$ by Proposition~\ref{PropBig}, and the latter is $\lhd$-stable, it
follows from the above computations that also $X_r$, $Y_r$ and $T_r$ are in
$B_{\lambda}^{\fin}(\g;\varepsilon,+)$.
\end{proof}
\begin{Proposition}
\label{PropErg}
The only $\lhd$-invariant elements in $B_{\lambda}^{\fin}(\g;\varepsilon,+)$ are scalar multiples of the
unit element.
\end{Proposition}
\begin{proof}
Assume that $x\in U_q(\g;\varepsilon,+)_{\fin}$ with $\pi_{\lambda}^{\varepsilon,+}(x)$ invariant.
As $U_q(\g;\varepsilon,+)_{\fin}$ is a~semi-simple right $U_q(\g)$-module, we have an equivariant
projection $E$ of $U_q(\g;\varepsilon,+)_{\fin}$ onto the $^*$-algebra of its invariant elements.
The latter is simply the center $\ZZ(U_q(\g;\varepsilon,+))$ of $U_q(\g;\varepsilon,+)$, by
Lemma~\ref{LemCenter}.
As $\pi_{\lambda}^{\varepsilon,+}$ is $\lhd$-equivariant by construction, we deduce that
$\pi_{\lambda}^{\varepsilon,+}(x) = \pi_{\lambda}^{\varepsilon,+}(E(x))$.
But the latter is a~scalar.
\end{proof}
\begin{Remark}
An alternative proof consists in applying Schur's lemma to the simple mo\-du\-le~$V_{\lambda}^{\varepsilon,+}$.
Indeed, $x\in B_{\lambda}^{\fin}(\g;\varepsilon,+)$ is $\lhd$-invariant if and only if it commutes with
all $\pi_{\lambda}^{\varepsilon,+}(y)$ for $y\in U_q(\g;\varepsilon,+)$.
As $V_{\lambda}^{\varepsilon,+}$ is simple, Schur's lemma implies that the algebra of $\lhd$-invariant
elements in $B_{\lambda}^{\fin}(\g;\varepsilon,+)$ forms a~field of countable dimension over $\C$,
hence coincides with~$\C$.
(I~would like the referee for pointing out this approach).
\end{Remark}
\begin{Proposition}
\label{PropCC}
Let $(V,\pi)$ be a~$^*$-representation of $B_{\lambda}^{\fin}(\g;\varepsilon,+)$ on a~pre-Hilbert space.
Then~$\pi$ is bounded.
\end{Proposition}

The proof is based on an argument which is well-known in the setting of compact quantum groups.
\begin{proof}
As $B_{\lambda}^{\fin}(\g;\varepsilon,+)$ consists of locally finite elements, any $b \in
B_{\lambda}^{\fin}(\g;\varepsilon,+)$ can be written as a~finite linear combination of elements $b_i
\in B_{\lambda}^{\fin}(\g;\varepsilon,+)$ for which there exists a~finite-di\-men\-sio\-nal
$^*$-representation $\pi$ of $U_q(\g)$ on a~Hilbert space such that $b_i\lhd h = \sum_j
\pi_{ij}(h)b_j$ for all $h\in U_q(\g)$, the $\pi_{ij}$ being the matrix components with respect to some
orthogonal basis.
An easy computation shows that $\sum_i b_i^*b_i$ is an invariant element, hence a~scalar by
Proposition~\ref{PropErg}.
Hence there exists $C\in \R^+$ such that for any $\xi\in V$ and any $i$, we have $\|\pi(b_i) \xi\|\leq
C\|\xi\|$.
We deduce that the element $\pi(b)$ is bounded.
\end{proof}
\begin{Definition}
A $B_{\lambda}^{\fin}(\g;\varepsilon,+)$-module $V$ is called a~\emph{highest weight module} if there
exists a~cyclic vector $v\in V$ which is annihilated by all $X_r$ and which is an eigenvector for all $Z_r$
with non-zero eigenvalue.
A pre-Hilbert space structure on $V$ is called \emph{invariant} if $\langle x\xi,\eta\rangle = \langle
\xi,x^*\eta\rangle$ for all $\xi,\eta\in V$ and $x\in B_{\lambda}^{\fin}(\g;\varepsilon,+)$.
\end{Definition}

We aim to show that the $B_{\lambda}^{\fin}(\g;\varepsilon,+)$ have only a~finite number of
non-equivalent irreducible highest weight modules.
Of course, each $B_{\lambda}^{\fin}(\g;\varepsilon,+)$ admits at least the highest weight mo\-du\-le~$V_{\lambda}^{\varepsilon,+}$.
Also note that, by an easy argument, each highest weight module decomposes into a~direct sum of joint
weight spaces for the~$Z_r$.
\begin{Proposition}
\label{PropFin}
Each $B_{\lambda}^{\fin}(\g;\varepsilon,+)$ admits only a~finite number of non-equivalent irreducible
highest weight modules.
\end{Proposition}
\begin{proof}
As the statement does not depend on the $^*$-structure, we may by rescaling restrict to the case that
$\varepsilon_r \in \{0,1\}$ for all $r$ upon allowing $\lambda \in \Char_{\C}(P)$.

By Proposition~\ref{PropBig} and the fact that any highest weight module is semi-simple for the torus part,
it is easily argued that any irreducible highest weight module of $B_{\lambda}^{\fin}(\g;\varepsilon,+)$
is obtained by restriction of a~$U_q(\g;\varepsilon,+)$-module $V_{\lambda'}^{\varepsilon,+}$ for some
$\lambda'\in \Char_{\C}(P)$.
As the center of $U_q(\g;\varepsilon,+)$ acts by the same character on $V_{\lambda}^{\varepsilon,+}$ and
$V_{\lambda'}^{\varepsilon,+}$, we find by Lemma~\ref{LemHarish} that the expression $\sum\limits_{w\in
W} \varepsilon_{\omega-w\omega}q^{-2(w\omega,\rho)}\lambda_{-4w\omega}$ remains the same upon replacing
$\lambda$ by $\lambda'$, for each $\omega\in P^+$.
Writing~$S$ for the set of $r$ with $\varepsilon_r=0$, it follows as in the proof of Lemma~\ref{LemHarish}
that
\begin{gather*}
\sum_{w\in W_S}q^{-2(w\omega,\rho)}\lambda_{-4w\omega}=\sum_{w\in W_S}q^{-2(w\omega,\rho)}
\lambda_{-4w\omega}'
\end{gather*}
for all $\omega\in P^{++}$, the strictly dominant weights.
As (invertible) characters on a~commutative semi-group are linearly independent, and as $P^{++}-P^{++} =
P$, it follows that the functions $\omega \rightarrow q^{-2(\omega,\rho)}\lambda_{-4\omega}$ and $\omega
\rightarrow q^{-2(\omega,\rho)}\lambda_{-4\omega}'$ on $P$ lie in the same $W_S$-orbit.
As the highest weight vector in an irreducible highest weight module is uniquely determined up to a~scalar,
and as the equivalence classes of such highest weight modules are then determined by the associated
eigenvalue of the $Z_r = \pi_{\lambda}^{\varepsilon,+}(K_{\omega_r}^{-4})$, this is sufficient to prove
the proposition.
\end{proof}

Remark that the above proof also gives the upper bound $|W_S|$ for the number of inequivalent highest
weight representations, but of course this estimate is not sharp if one only considers unitarizable
representations.
\begin{Proposition}
\label{PropCCC}
Let $\pi$ be a~$^*$-representation of $B_{\lambda}^{\fin}(\g;\varepsilon,+)$ on a~Hilbert space $\mathscr{H}$.
If $0$ is not in the point-spectrum of any of the $Z_r$, then $\mathscr{H}$ is a~$($possibly infinite$)$
direct sum of completions $(\Hsp_k,\pi_k)$ of unitarizable highest weight modules of
$B^{\fin}_{\lambda}(\g;\varepsilon,+)$.
\end{Proposition}

\begin{proof}
By a~direct integral decomposition, and using Proposition~\ref{PropFin}, it is sufficient to show that
any such \emph{irreducible} $^*$-representation $\pi$ of $B_{\lambda}^{\fin}(\g;\varepsilon,+)$ on
a~Hilbert space $\mathscr{H}$ is the completion of a~highest weight module with invariant pre-Hilbert space
structure.
We then argue as in~\cite[Section~3]{Mas1}.
Write $\chi_{X}$ for the characteristic function of a~set.
By assumption, there exists $t\in \R^l$ with $t_r\neq 0$ for all $r$ and $P_{t} = \chi_{\prod_r
\lbrack q_rt_r,t_r\rbrack}(\pi(Z_1),\ldots,\pi(Z_l))$ non-zero.
Suppose now that~$r$ is such that $\pi(X_r)P_{t}\neq 0$.
From the commutation relations between the~$X_r$ and the~$Z_s$, we deduce that
$P_{(t_1,\ldots,q_r^{-2}t_r,\ldots,t_l)}\neq 0$.
As the $\pi(Z_r)$ are bounded, this process must necessarily stop.
Hence we may choose $t$ such that $P_{t}\neq 0$ but $\pi(X_r)P_{t}=0$ for all $r$.

Let $V$ be the union of the images of the spectral projections of $(Z_1,\ldots, Z_l)$ corresponding to the
$\prod_r \left(\R\setminus (-\frac{1}{n},\frac{1}{n})\right)$ with $n\in \N$.
As $B_{\lambda}^{\fin}(\g;\varepsilon,+)$ is spanned by elements which skew-commute with the $Z_r$, it
follows that $V$ is a~$B_{\lambda}^{\fin}(\g;\varepsilon,+)$-module on which the $\pi(Z_r)$ are
invertible linear maps.
This entails that the restriction of $\pi$ to $V$ can be extended to a~representation $\widetilde{\pi}$ of
$B^{\fin}(\g;\varepsilon,+)_{\ext}$, the sub$^*$-algebra of $B(\g;\varepsilon,+)$ generated by $X_r$, $Y_r$
and the $Z_r^{\pm 1}$ (which contains $B_{\lambda}^{\fin}(\g;\varepsilon,+)$ by
Proposition~\ref{PropBig}).
Note that this $^*$-algebra admits a~triangular decomposition (in the obvious way with respect to the above
generators).

Pick now a~non-zero $\xi \in P_{t}\mathscr{H}$.
Suppose that $\xi$ were not in the pure point spectrum of some $\pi(Z_r)$.
Then we can find $q_rt_r<a<t_r$ such that $\chi_{\lbrack q_rt_r,a\rbrack}(\pi(Z_r))\xi\neq 0 \neq
\chi_{(a,t_r\rbrack}(\pi(Z_r))\xi$.
However, $\lbrack q_1t_1,t_1\rbrack \times \cdots \times \lbrack q_rt_r,a\rbrack\times \cdots \times
\lbrack q_lt_l,t_l\rbrack \cap \prod_s \lbrack q_s^{2k_s+1}t_s,q_s^{2k_s}t_s\rbrack = \varnothing$ for
all $k_s\in \mathbb{N}$ with at least one $k_s>0$.
From the commutation relations between the $Y_s$ and $Z_{s'}$, and the fact that $\pi(X_s)\xi = 0$ for all
$s$, we deduce that $\chi_{\lbrack q_rt_r,a\rbrack}(\pi(Z_r))\xi$ is orthogonal to the
$B_{\lambda}^{\fin}(\g;\varepsilon,+)_{\ext}$-module spanned by $\chi_{(a,t_r\rbrack}(\pi(Z_r))\xi$.
As $\pi$ is irreducible, this would entail $\chi_{\lbrack q_rt_r,a\rbrack}(\pi(Z_r))\xi=0$.
Having arrived at a~contradiction, we conclude that $\xi$ is a~joint eigenvector of all $\pi(Z_r)$.

As $\xi$ is annihilated by all $\pi(X_r)$ and is a~joint eigenvector of all $\pi(Z_r)$, the module
generated by it is a~highest weight module.
As $\pi$ was irreducible, this module must necessarily be dense in $\mathscr{H}$, and the proposition is proven.
\end{proof}

We now want to consider analytic versions of the $B^{\fin}_{\lambda}(\g;\varepsilon,+)$.
\begin{Definition}
Let $B$ be a~unital $^*$-algebra.
We say that $B$ admits a~\emph{universal $C^*$-envelope} if there exists a~non-trivial unital $C^*$-algebra
$C$ together with a~unital $^*$-homomorphism $\pi_u\colon B\rightarrow C$ of unital $^*$-algebras such that any
$^*$-homomorphism $B\rightarrow D$ with $D$ a~unital $C^*$-algebra factors through $C$.
\end{Definition}

Of course, the above $C^*$-algebra $C$ is then uniquely determined up to isomorphism.
\begin{Definition}
We define $\Pol(\G_+^q)$ to be the Hopf $^*$-algebra inside the dual of $U_q(\g)$ which is spanned by the
matrix coefficients of finite-di\-men\-sio\-nal highest weight representations of $U_q(\g)$ associated to
positive characters.
We define
\begin{gather*}
\alpha_{\lambda}^{\varepsilon,+}\colon  \ B_{\lambda}^{\fin}
(\g;\varepsilon,+)\rightarrow\Pol(\G_+^q)\otimes B_{\lambda}^{\fin}(\g;\varepsilon,+)
\end{gather*}
as the comodule $^*$-algebra structure dual to the module $^*$-algebra structure $\lhd$ by $U_q(\g)$.
\end{Definition}

Note that the latter definition makes sense, since $B_{\lambda}^{\fin}(\g;\varepsilon,+)$ is integrable
as a~right $U_q(\g)$-module.

It is known~\cite{Lev1} that $\Pol(\G_+^q)$ admits a~universal $C^*$-algebraic envelope $C(\G_+^q)$, which
becomes a~compact quantum group in the sense of~\cite{Wor1}.
We will denote by $\varphi_{\G_+^q}$ the invariant state on $C(\G_+^q)$, which is faithful by
co-amenability of $\G_+^q$.
\begin{Lemma}
\label{LemEnv}
Assume that $B_{\lambda}^{\fin}(\g;\varepsilon,+)$ admits at least one $^*$-representation on a~Hilbert
space.
Then $B_{\lambda}^{\fin}(\g;\varepsilon,+)$ admits a~$C^*$-algebraic envelope
$C_{\lambda}(\g;\varepsilon,+)$.
\end{Lemma}
\begin{proof}
The universal $C^*$-algebraic envelope of $B_{\lambda}^{\fin}(\g;\varepsilon,+)$ exists for precisely the
same reason as in Proposition~\ref{PropCC}, since for any element $b\in
B_{\lambda}^{\fin}(\g;\varepsilon,+)$ there exists a~universal constant~$C_b$ such that $\|\pi(b)\|\leq
C$ for all $^*$-representations $\pi$ of $B_{\lambda}^{\fin}(\g;\varepsilon,+)$ by bounded operators on
a~Hilbert space.
As $B_{\lambda}^{\fin}(\g;\varepsilon,+)$ admits at least one $^*$-representation, we have
$C_{\lambda}(\g;\varepsilon,+)\neq 0$.
\end{proof}
\begin{Remark}
If $V_{\lambda}^{\varepsilon,+}$ is unitarizable, it is of course clear that we get a~\emph{faithful} map
from $B_{\lambda}^{\fin}(\g;\varepsilon,+)$ into $C_{\lambda}(\g;\varepsilon,+)$.
\end{Remark}
\begin{Lemma}
Let $C_{\lambda}(\g;\varepsilon,+)$ be the universal $C^*$-envelope of
$B_{\lambda}^{\fin}(\g;\varepsilon,+)$, whenever it exists.
Then $\alpha_{\lambda}^{\varepsilon,+}$ induces a~$C^*$-algebraic coaction by $C(\G_+^q)$ on
$C_{\lambda}(\g;\varepsilon,+)$.
\end{Lemma}
\begin{proof}
The map $\alpha_{\lambda}^{\varepsilon,+}$ gives a~$C^*$-representation of
$B^{\fin}_{\lambda}(\g;\varepsilon,+)$ into $C(\G_+^q)\otimes C_{\lambda}(\g;\varepsilon,+)$, which hence
factors over $C_{\lambda}(\g;\varepsilon,+)$.
It is straightforward to argue that this is a~$C^*$-algebraic coaction.\looseness=1
\end{proof}

As $B_{\lambda}^{\fin}(\g;\varepsilon,+)$ only had scalar multiples of the unit as invariants, it follows
that the coaction~$\alpha_{\lambda}^{\varepsilon,+}$ on $C_{\lambda}(\g;\varepsilon,+)$ is
\emph{ergodic}~\cite{Boc1}, i.e.\
if $\alpha_{\lambda}^{\varepsilon,+}(x) = 1\otimes x$, then $x\in \mathbb{C}1$.
We will write $\varphi_{\lambda}^{\varepsilon,+}$ for the unique invariant state on
$C_{\lambda}(\g;\varepsilon,+)$, so
\begin{gather*}
\big(\varphi_{\G_+^q}\otimes\iota\big)\alpha_{\lambda}^{\varepsilon,+}(x)=\varphi_{\lambda}^{\varepsilon,+}
(x)1,
\qquad
\forall\, x\in C_{\lambda}(\g;\varepsilon,+).
\end{gather*}
As $\varphi_{\G_+^q}$ is faithful, also $\varphi_{\lambda}^{\varepsilon,+}$ is faithful.
\begin{Notation}
We write $\theta_{\lambda,\reg}^{\varepsilon,{+}}$ for the GNS-representation of
$(C_{\lambda}(\g;\varepsilon,\!{+}),\varphi_{\lambda}^{\varepsilon,{+}})$, and $W_{\lambda}(\g;\varepsilon,\!{+})$
for the von Neumann algebraic completion of $C_{\lambda}(\g;\varepsilon,\!{+})$ in this GNS-representation.
\end{Notation}

From Lemma~\ref{LemEnv}, it follows that a~$C^*$-algebraic envelope of
$B_{\lambda}^{\fin}(\g;\varepsilon,+)$ exists if $V_{\lambda}^{\varepsilon,+}$ is unitarizable.
In this case, we can say something explicit about $W_{\lambda}(\g;\varepsilon,+)$.
\begin{Theorem}
\label{TheoFI}
Assume that $V_{\lambda}^{\varepsilon,+}$ is unitarizable.
Then $W_{\lambda}(\g;\varepsilon,+)$ is a~finite direct sum of type $I$ factors.
\end{Theorem}

The proof will make use of the following standard lemma.
\begin{Lemma}
\label{LemExtvN}
Let $A$ be a~unital $C^*$-algebra with faithful state $\varphi$.
Let $M$ be the von Neumann algebra closure of $A$ in its GNS-representation with respect to $\varphi$.
Let $\pi$ be a~representation of~$A$ on a~Hilbert space $\mathscr{H}$ such that there exists a~faithful
state $\omega \in B(\mathscr{H})_*$ with $\omega \circ \pi = \varphi$.
Then~$\pi$ extends to a~normal faithful $^*$-representation of~$M$.
\end{Lemma}

\begin{proof}
As $\omega$ is faithful, the bicommutant $\pi(A)''$ is faithfully represented on the GNS-space
$\mathscr{L}^2(\pi(A),\omega)$.
The unitary $U\colon \mathscr{L}^2(A,\varphi)\rightarrow \mathscr{L}^2(\pi(A),\omega)$ induced by $\pi$ then
provides an isomorphism $M\rightarrow \pi(A)''$ extending $\pi$.
\end{proof}
\begin{proof}[Proof of Theorem~\ref{TheoFI}] As mentioned, the $C^*$-algebraic envelope
$C_{\lambda}(\g;\varepsilon,+)$ certainly  exists.

To prove the remaining part of the theorem, we first make some preparations.
Recall that the dual of $U_q(\ssl(2,\C))$ can be identified with $\Pol(SU_{q}(2))$, Woronowicz's twisted
quantum $SU(2)$-group~\cite{Wor2}.
It is well-known that its von Neumann algebraic envelope $\mathscr{L}^{\infty}(SU_q(2))$ is isomorphic to
$B(l^2(\N))\otimes \mathscr{L}(\Z)$, and as such admits a~faithful representation on the Hilbert space
$\Hsp_+ = l^2(\N)\otimes l^2(\Z)$ (cf.~\cite{Lev1,Res1}).

More generally, write $U_{q_r}(\su(2))$ for the sub-Hopf-$^*$-algebra generated by the $X_r^{\pm}$ and
$K_{\alpha_r}^{\pm1}$ inside $U_{q}(\g)$.
By duality, one obtains a~surjective $^*$-homomorphism $\gamma_r\colon  \Pol(\G_+^q)\rightarrow
\Pol(SU_{q_r}(2))$.
This induces a~$^*$-representation of $\Pol(\G_+^q)$ on $\Hsp_+$, which we will denote by the same symbol~$\gamma_r$.
Suppose now that $t=(r_1,\ldots,r_n)$ is an ordered $n$-tuple of elements in~$I$.
Then we obtain a~$^*$-representation of $\Pol(G_+^q)$ on $\Hsp_+^{\otimes n}$ by means of the
$^*$-homomorphism
\begin{gather*}
\gamma_t=(\gamma_{r_1}\otimes\cdots\otimes\gamma_{r_n})\circ\Delta_{\G_+^q}^{(n)},
\end{gather*}
where $\Delta_{\G_+^q}^{(n)}$ denotes the $n$-fold coproduct.
Let now $t_0 = (r_1,\ldots,r_N)$ be such that $w_0 = s_{r_1}\cdots s_{r_N}$ is a~reduced expression for the
longest element in the Weyl group of $\g$.
By~\cite{Res1}, we know that~$\varphi_{\G_+^q}$ can be realized as $\omega \circ \gamma_{t_0}$ for some
faithful normal state $\omega \in B\big(\Hsp_+^{\otimes N}\big)_*$.
By Lemma~\ref{LemExtvN}, this implies that~$\gamma_{t_0}$ can be extended to a~faithful normal
$^*$-representation of $\mathscr{L}^{\infty}(G_+^q)$.

Let us turn now to the proof of the theorem.
The main step is to prove that the point-spectrum of
$\theta_{\lambda,\reg}^{\varepsilon,+}\left(\prod\limits_{r=1}^l Z_r\right)$ does not contain zero.
Indeed, if this is the case, then we can use Proposition~\ref{PropCCC} to conclude that
$\theta_{\lambda,\reg}^{\varepsilon,+}$ decomposes into a~direct integral of highest weight modules for
$B_{\lambda}^{\fin}(\g;\varepsilon,+)$.
It then follows that $W_{\lambda}(\g;\varepsilon,+)$ is simply $\oplus_{k=1}^m B(\Hsp_k)$ with $\Hsp_k$ the
Hilbert space completion of the highest weight modules which appear in
$\theta_{\lambda,\reg}^{\varepsilon,+}$.

To show that $\theta_{\lambda,\reg}^{\varepsilon,+}\left(\prod\limits_{r=1}^l Z_r\right)$ does not contain zero in its
pointspectrum, it is sufficient to show that the operator
$(\gamma_{t_0}\otimes\theta_{\lambda}^{\varepsilon,+})\left(\alpha_{\lambda}^{\varepsilon,+}
\left(\prod\limits_{r=1}^l Z_r\right)\right)$
does not contain zero in its pointspectrum, where $\theta_{\lambda}^{\varepsilon,+}$
is the $^*$-representation of $B_{\lambda}^{\fin}(\g;\varepsilon,+)$ on the Hilbert space completion
$\Hsp_{\lambda}^{\varepsilon,+}$ of $V_{\lambda}^{\varepsilon,+}$.
Indeed, the invariant state $\varphi_{\lambda}^{\varepsilon,+} = (\varphi_{\G_+^q}\otimes
\id)\alpha_{\lambda}^{\varepsilon,+}$ can be extended to a~faithful normal functional on
$B(\Hsp_{\lambda}^{\varepsilon,+}\otimes \mathscr{H_+}^{\otimes N})$, which implies, again by
Lemma~\ref{LemExtvN}, that
$(\gamma_{t_0}\otimes\theta_{\lambda}^{\varepsilon,+})\alpha_{\lambda}^{\varepsilon,+}$ extends to
a~faithful normal $^*$-representation of $W_{\lambda}(\g;\varepsilon,+)$.

Finally, to show that $(\theta_{\lambda}^{\varepsilon,+}\otimes
\gamma_{t_0})\left(\alpha_{\lambda}^{\varepsilon,+}\left(\prod\limits_{r=1}^l Z_r\right)\right)$
does not contain zero in its pointspectrum, we can reason by induction, using the following lemma.
\end{proof}
\begin{Lemma}
\label{LemNoZero}
Let $(V,\pi)$ be an irreducible highest weight module for $B_{\lambda}^{\fin}(\g;\varepsilon,+)$ with an
invariant pre-Hilbert space structure, and let $\Hsp$ be the completion of $V$.
Fix $r\in I$, and put $\pi_r = (\pi\otimes \gamma_r)\alpha_{\lambda}^{\varepsilon,+}$.
Then $\pi_r\left(\prod\limits_{s=1}^l Z_s\right)$ does not contain $0$ in its point spectrum.
\end{Lemma}
\begin{proof}
It is easy to see that $\pi_r(Z_s) = 1\otimes \pi(Z_s)$ for $r\neq s$, so that none of these operators have
zero in their point-spectrum.
Let now $A_r$ be the sub-$^*$-algebra of $B_{\lambda}^{\fin}(\g;\varepsilon,+)$ generated by
$Z_r$, $T_r$, $X_r$, $Y_r$, $W_r$, see Lemma~\ref{LemSubAlg}.
By that lemma, $A_r$ is stable under the right action by~$U_{q_r}(\su(2))$, and~$W_r$ and $T_r$ are
invariants in the center of~$A_r$.
By invariance, $\pi_r(W_r) = 1\otimes \pi(W_r)$ and $\pi_r(T_r) = 1\otimes \pi(T_r)$ are bounded
self-adjoint operators.
Hence, to investigate the spectrum of~$\pi_r(Z_r)$, we may by disintegration treat the above operators as
scalars, say~$w_r$ and~$t_r$.
Denote the resulting quotient of $A_r$ by $A_r(w_r,t_r)$.

From Lemma~\ref{LemSubAlg}, we find commutation relations between the generators $X_r$, $Y_r$ and $Z_r$ of
$A_r(w_r,t_r)$, as well as the resulting action of $U_{q_r}(\su(2))$.
It follows that $A_r(w_r,t_r)$ is an equivariant quotient of a~generalized Podle\'{s} sphere
$S_{q_r,\tau}^2$ for $SU_{q_r}(2)$~\cite{Mas1}, for some $\tau$ depending on~$t_r$ and~$w_r$.
Moreover, as $\varphi_{SU_{q_r}(2)}$ can be realized on the Hilbert space $\Hsp_+$, the von Neumann
algebraic envelope of $A_r(w_r,t_r)$ will be isomorphic to the von Neumann algebraic envelope of~$S_{q_r,\tau}^2$, which is equal to $M_n(\C)$, $B(l^2(\N))$ or $B(l^2(\N))\oplus B(l^2(\N))$, depending on
whether $\varepsilon_rw_r$ is positive, zero or negative (cf.~\cite{Mas1}).
In any case, the corresponding image of~$Z_r$ will not contain~0 in its point-spectrum.
\end{proof}

\section[More on the $\varepsilon \in \Char_{\{0,1\}}(Q^+)$-case]{More on the $\boldsymbol{\varepsilon \in
\Char_{\{0,1\}}(Q^+)}$-case}\label{section4}

The case of the non-standard Podle\'{s} spheres already shows that $W_{\lambda}(\g;\varepsilon,+)$ is in
general not a~factor.
However, for $\varepsilon\in \Char_{\{0,1\}}(Q^+)$ and $\lambda\in \Char_{\R^+}(P)$ such that
$W_{\lambda}(\g;\varepsilon,+)$ is well-defined, we show that $W_{\lambda}(\g;\varepsilon,+)$ \emph{does}
become a~type $I$-factor, and we can then also say something more about the invariant integral
$\varphi_{\lambda}^{\varepsilon,+}$ on $W_{\lambda}(\g;\varepsilon,+)$.
\begin{Proposition}
\label{PropFac}
Let $\varepsilon\in \Char_{\{0,1\}}(Q^+)$, $\lambda\in \Char_{\R^+}(P)$, and suppose
$V_{\lambda}^{\varepsilon,+}$ is unitarizable with completion $\Hsp_{\lambda}^{\varepsilon,+}$.
Identify $B_{\lambda}^{\fin}(\g;\varepsilon,+)\subseteq B(\Hsp_{\lambda}^{\varepsilon,+})$.
Then this inclusion completes to a~natural identification $W_{\lambda}(\g;\varepsilon,+) \cong
B(\Hsp_{\lambda}^{\varepsilon,+})$.
\end{Proposition}
\begin{proof}
From the proof of Theorem~\ref{TheoFI} and Lemma~\ref{LemNoZero}, and from the commutation relations in
Lemma~\ref{LemSubAlg}, we get that the $A_r(w_r,t_r)$ appearing in the proof of Lemma~\ref{LemNoZero} can
only be matrix algebras or standard Podle\'{s} spheres.
As the $Z_r$ in the von Neumann algebraic completion of these algebras are always positive operators, it
follows by induction that the components appearing in $W_{\lambda}(\g;\varepsilon,+)$ arise from
restrictions of highest weight modules $V_{\lambda'}^{\varepsilon,+}$ of $U_q(\g;\varepsilon,+)$ with
$\lambda'\in \Char_{\R^+}(P)$.
Our aim is to show that necessarily $\lambda'=\lambda$.

Let $S$ be the set of $r$ with $\varepsilon_r=1$.
Suppose that $\lambda'\in \Char_{\R^+}(P)$ is such that the representation of $U_q(\g;\varepsilon,+)$ on
$V_{\lambda'}^{\varepsilon,+}$ factors over $B_{\lambda}^{\fin}(\g;\varepsilon,+)$.
Suppose that $V_{\lambda'}^{\varepsilon,+}$ then admits an invariant pre-Hilbert space structure as
a~$B_{\lambda}^{\fin}(\g;\varepsilon,+)$-module, hence as a~$U_q(\g;\varepsilon,+)$-module by
Proposition~\ref{PropBig}.
From the proof of Proposition~\ref{PropFin}, we deduce that there exists $w\in W_S$ such that
$q^{(-2w\omega,\rho)}\lambda_{-4w\omega} = q^{(-2\omega,\rho)}\lambda_{-4\omega}'$ for all $\omega\in P$.
In particular, $\lambda_{\omega_r} = \lambda_{\omega_r}'$ for $r\notin S$.

On the other hand, let $\widetilde{U}_q(\kk_{\sss})$ be the subalgebra of $U_q(\g;\varepsilon,+)$ generated
by the~$X_r^{\pm}$ and~$K_{\alpha_r}^{\pm 1}$ with~$r\in S$.
Let $\widetilde{V}_{\lambda}^{\varepsilon,+}$ be the $\widetilde{U}_q(\kk_{\sss})$-module spanned by
$v_{\lambda}^{\varepsilon,+}$, and similarly for $V_{\lambda'}^{\varepsilon,+}$.
Then by Proposition~\ref{PropPosDef}, these are irreducible highest weight modules for $U_q(\kk_{\sss})$
associated to the restrictions of $\lambda$ and $\lambda'$ to the root lattice $Q_S$ of $\kk_{\sss}$.
But as $V_{\lambda'}^{\varepsilon,+}$ admits an invariant pre-Hilbert space structure (and $\kk_{\sss}$ is
compact), it is necessarily finite-di\-men\-sio\-nal.
However, as the restriction of $\lambda'$ lies in the $W_S$-orbit of $\lambda$ (for the so-called
`dot'-action), it is well-known that this can happen only if the restrictions of $\lambda$ and $\lambda'$
to $Q_S$ coincide.
Combined with the observation at the end of the previous paragraph, this forces $\lambda = \lambda'$ inside
$\Char_{\R^+}(P)$.
\end{proof}
\begin{Proposition}
Under the assumptions of Proposition~{\rm \ref{PropFac}}, the invariant state on $W_{\lambda}(\g;\varepsilon,+)$
is given by
\begin{gather*}
\varphi_{\lambda}^{\varepsilon,+}(x)=\frac{\Tr(xZ_{\rho})}{\Tr(Z_{\rho})},
\end{gather*}
where $Z_{\rho} = \prod\limits_{r=1}^l Z_r$.
\end{Proposition}
\begin{proof}
Consider the projection $E^{\varepsilon,+}$ of $U_q(\g;\varepsilon,+)_{\fin}$ onto its direct summand
$\ZZ(U_q(\g;\varepsilon,+))$, the space of $\lhd$-invariants.
For $\varepsilon_r>0$ for all $r$, it follows from~\cite[Chapter 7]{Jos1} that $E^{\varepsilon,+}(xy) =
E^{\varepsilon,+}(y\sigma(x))$ for all $x,y\in U_q(\g;\varepsilon,+)_{\fin}$, where $\sigma(x) =
K_{-4\rho}xK_{-4\rho}^{-1}$.
By continuity, this then holds for all $\varepsilon\in \Char_{\R}(Q^+)$.
Clearly $\sigma$ induces an automorphism of $B_{\lambda}^{\fin}(\g;\varepsilon,+)$, which we will denote
by the same symbol.
As $\varphi_{\lambda}^{\varepsilon,+}$ factors over $E^{\varepsilon,+}$, we have as well that
$\varphi_{\lambda}^{\varepsilon,+}(xy) = \varphi_{\lambda}^{\varepsilon}(y\sigma(x))$.

However, by general theory~\cite{Boc1} we know that the modular automorphism group $\sigma_t$ of
$\varphi_{\lambda}^{\varepsilon,+}$ leaves $B_{\lambda}^{\fin}(\g;\varepsilon,+)$ invariant and is
diagonalizable (and hence analytic) on it.
From the above, we can hence conclude that $\sigma_t = \Ad(Z_{\rho}^{it})$ (where we use that $Z_{\rho}$ is
a~positive operator by Proposition~\ref{PropFac}).
As $W_{\lambda}(\g;\varepsilon,+)$ is a~type $I$-factor, again by Proposition~\ref{PropFac}, we know that
$\varphi_{\lambda}^{\varepsilon,+}$ is completely determined by its modular automorphism group up to a~scalar.
Hence we obtain our expression for $\varphi_{\lambda}^{\varepsilon,+}$ as in the statement of the
proposition.
\end{proof}

\appendix

\section{Cogroupoids}
\label{SecCog}

In this section, we recall the notion of cogroupoid due to J.~Bichon~\cite{Bic1,Bic2}
(cf.~also the notion of \emph{face algebra}~\cite{Hay1}).
\begin{Definition}%\label{DefWHo}
Let $I$ be an index set, and let $\{H_{ij}\,|\, i,j\in I\}$ be a~collection of $^*$-algebras.
Suppose that for each triple of indices $i,j,k\in I$, we are given a~unital $^*$-homomorphism
\begin{gather*}
\Delta_{ij}^k\colon  \ H_{ij}\rightarrow H_{ik}\otimes H_{kj},
\qquad
h\mapsto h_{(1)ik}\otimes h_{(2)kj}.
\end{gather*}
We call $\big(\{H_{ij}\},\{\Delta_{ij}^k\}\big)$ a~connected \emph{cogroupoid} over the index set $I$ if the
following conditions are satisfied:
\begin{itemize}\itemsep=0pt
\item[$\bullet$](Connectedness) None of the $H_{ij}$ are the zero algebra.
\item[$\bullet$](Coassociativity) For each quadruple $i$, $j$, $k$, $l$ of indices, we have
\begin{gather*}
\big(\id\otimes\Delta_{jl}^k\big)\Delta_{il}^j=\big(\Delta_{ik}^j\otimes\id\big)\Delta_{il}^k,
\end{gather*}
\item[$\bullet$](Counits) There exist unital $^*$-homomorphisms $\epsilon_i\colon H_{ii}\rightarrow \C$ such
that, for all indices~$i$,~$j$,
\begin{gather*}
(\epsilon_i\otimes\id)\Delta_{ij}^i=\id_{H_{ij}}=(\id\otimes\epsilon_j)\Delta_{ij}^j,
\end{gather*}
\item[$\bullet$](Antipodes) There exist anti-homomorphisms $S_{ij}\colon  H_{ij}\rightarrow H_{ji}$ such that,
for all indices~$i$,~$j$ and all $h\in H_{ii}$, we have
\begin{gather*}
S_{ij}(h_{(1)ij})h_{(2)ji}=\epsilon_i(h)=h_{(1)ij}S_{ji}(h_{(2)ji}).
\end{gather*}
\end{itemize}

\end{Definition}

As for Hopf $^*$-algebras, it is easy to show that the $S_{ij}$ are unique, and that $S_{ji}(S_{ij}(h)^*)^*
= h$ for each $h\in H_{ij}$.
Note that each $\big(H_{ii},\Delta_{ii}^i\big)$ defines a~Hopf $^*$-algebra.
\begin{Definition}
\label{DefRA}
Let $\big(H_{ij},\Delta_{ij}^k\big)$ be a~cogroupoid.
The \emph{right adjoint action} (or Miyashita--Ulbrich action) $\lhd$ of $H_{jj}$ on $H_{ij}$ is given by
\begin{gather*}
x\lhd h=S_{ji}(h_{(1)ji})xh_{(2)ij}.
\end{gather*}
\end{Definition}

One easily proves that $\lhd$ defines a~right $H_{jj}$-module $^*$-algebra structure on $H_{ij}$.
The compatibility with the $^*$-structure means that $(x\lhd h)^* = x^*\lhd S_{jj}(h)^*$.
\begin{Lemma}
\label{LemCenter}
The space of $\lhd$-invariant elements in $H_{ij}$ coincides with the center of $H_{ij}$.
\end{Lemma}
\begin{proof}
(Cf.~\cite[Lemma 2.4]{Jos1}.) If $x\in H_{ij}$ is in the center, clearly $x\lhd h = \varepsilon_j(h)x$ for
all $h\in H_{jj}$, by definition of the antipode.
Conversely, if $S_{ji}(h_{(1)ji})xh_{(2)ij} = \varepsilon_j(h)x$ for all $h\in H_{jj}$, we have for $y\in
H_{ij}$ that $xy = y_{(1)ij}S_{ji}(y_{(2)ji})xy_{(3)ij} = yx$.
\end{proof}

Suppose now that $B$ is a~unital $^*$-algebra, and that for some $i$, $j$ we have a~unital $^*$-homo\-mor\-phism
$\pi\colon  H_{ij}\rightarrow B$.
As the right action of $H_{jj}$ on $H_{ij}$ is inner, it descends to a~right action on~$B$,
\begin{gather*}
b\lhd h=\pi(S_{ji}(h_{(1)ji}))b\pi(h_{(2)ij}).
\end{gather*}
Note that the central elements in $B$ are invariant for the action.

A particular case of cogroupoid can be constructed from a~Hopf algebra $(H,\Delta)$ together with
a~collection of real 2-cocycle functionals $\{\omega_i\,|\, i\in I\}$ on it.
Here we mean by real 2-cocycle functional an element $\omega\in (H\otimes H)^*$ which is convolution
invertible and such that $\omega(1,h) = \varepsilon(h) = \omega(h,1)$ for all $h\in H$, while
$\omega(h^*,k^*) = \overline{\omega(k,h)}$ and
\begin{gather*}
\omega(h_{(1)},k_{(1)})\omega(h_{(2)}k_{(2)},l)=\omega(k_{(1)},l_{(1)})\omega(h,k_{(2)}l_{(2)}),
\qquad
\forall\, h,k,l\in H.
\end{gather*}
Let us write $H_{ij}$ for the vector space $H$ with the new multiplication
\begin{gather*}
m_{ij}\colon  \ H\otimes H\rightarrow H,
\qquad
h\otimes k\mapsto\omega_i(h_{(1)},k_{(1)})h_{(2)}k_{(2)}\omega_j^{-1}(h_{(3)},k_{(3)}),
\end{gather*}
and write $\Delta_{ij}^k$ for the given coproduct $\Delta$ seen as a~map $H_{ij}\rightarrow H_{ik}\otimes
H_{kj}$.
Then the $(H_{ij},\Delta_{ij}^k)$ form a~connected cogroupoid.

\section{Continuous one-parameter families of Lie algebras}
\label{SecLieAlg}

We introduce the real Lie algebras whose quantizations we studied in Section~\ref{section1}.
We refer to standard works as~\cite{Hel1,Hum1,Kna1} for the basic background on Lie algebras and Lie groups.

We keep the notation as in Section~\ref{section1}.
We further write $\{h_{r}\} \subseteq \h$ for the basis dual to $\{\omega_r\}$, and write $\h_{\R}$ for its
real span, which we may identify with the real dual of $\h_{\R}^*$.
We write $\g_c\subseteq \g$ for the compact real form of $\g$, and $^\dag$ for the corresponding
anti-linear anti-automorphism such that $x^{\dag} = -x$ for $x\in \g_c$.

Finally, for each $\alpha\in \Delta^+$, we choose root vectors $X^{\pm}_{\alpha}\in \n^{\pm}$ such that
$(X_{\alpha}^{+})^\dag = X_{\alpha}^-$ for all $\alpha\in \Delta^+$ and $\lbrack X_r^+,X_r^-\rbrack = h_r$
for all $\alpha_r\in \Phi^+$.

Fix $\varepsilon\in \Char_{\R}(Q^+)$.
We can define on $\g$ the linear maps
\begin{gather*}
S_{\varepsilon}^{\pm}\colon \ \g\rightarrow\g\colon
\begin{cases}
X_{\alpha}^{\pm} \mapsto \varepsilon_{\alpha}X_{\alpha}^{\pm},& \alpha\in\Delta^+,
\\
z \mapsto  z,& z\in\h+\n^{\mp}.
\end{cases}
\end{gather*}
\begin{Definition}
We define $\g_{\varepsilon}$ to be the complex vector space
\begin{gather*}
\{(S_{\varepsilon}^+(z),S_{\varepsilon}^-(z))\,|\, z\in\g\}\subseteq\g\oplus\g.
\end{gather*}
\end{Definition}
\begin{Proposition}
The vector space $\g_{\varepsilon}$ is a~Lie $^*$-subalgebra of the direct sum Lie algebra $\g\oplus \g$
equipped with the involution $(w,z)^* = (z^{\dag},w^{\dag})$.
Moreover, $\dim_\C(\g_{\varepsilon}) = \dim_\C(\g)$.
\end{Proposition}
\begin{proof}
From the definition, we see that $\g_{\varepsilon}$ is the linear space generated by elements of the form
$(\varepsilon_{\alpha}X_{\alpha}^+,X_{\alpha}^+)$, $(X_{\alpha}^-,\varepsilon_{\alpha}X_{\alpha}^-)$ and
$(h_r,h_r)$.
Accordingly, $\g_{\varepsilon}$ has dimension $\dim_{\C}(\g)$ and inherits the $^*$-operation from
$\g\oplus \g$.
Since $\varepsilon_{\alpha+\beta}= \varepsilon_{\alpha}\varepsilon_{\beta}$ whenever $\alpha$, $\beta$ and
$\alpha+\beta$ are positive roots, we find that $\g_{\varepsilon}$ is closed under the bracket operation.
\end{proof}

In the following, we consider $\g_{\varepsilon}$ with its $^*$-operation inherited from $\g\oplus \g$.
\begin{Remark}
By rescaling the $X_{\alpha}^{\pm}$, we see that $\g_{\varepsilon} \cong \g_{\eta}$ as Lie $^*$-algebras
whenever $\varepsilon_r = \lambda_r\eta_r$ for certain $\lambda_r>0$.
Hence we may in principle always assume that $\varepsilon \in \{-1,0,1\}^l$.
It is however sometimes convenient to keep the continuous deformation aspect into the game.
In the physics literature, this type of deformation goes by the name of `the contraction method'.
(In low dimensions, it can easily be visualized, cf.~\cite[Chapter 13]{Gil1}.)
\end{Remark}
\begin{Definition}
We define $\g_{\varepsilon}^{\nc} = \{z\in \g_{\varepsilon}\,|\, z^* = -z\}$.
\end{Definition}

Hence $\g_{\varepsilon}^{\nc}$ is a~real Lie algebra with $\g_{\varepsilon}$ as its complexification.
We can also realize $\g^{\nc}_{\varepsilon}$ more conveniently inside $\g$ as follows.
\begin{Proposition}\label{PropIden}
Write
\begin{gather*}
X_{\alpha}^{(\varepsilon)}=X_{\alpha}^+-\varepsilon_{\alpha}X_{\alpha}^-, \qquad
Y_{\alpha}^{(\varepsilon)}=i\big(X_{\alpha}^++\varepsilon_{\alpha}X_{\alpha}^-\big)
\end{gather*}
as elements in $\g$.
Consider the $\R$-linear span of the $X_{\alpha}^{(\varepsilon)}$, $Y_{\alpha}^{(\varepsilon)}$ and $ih_{r}$.
Then this space is closed under the Lie bracket of $\g$, and forms a~real Lie algebra isomorphic to
$\g^{\nc}_{\varepsilon}$.
\end{Proposition}
\begin{proof}
As a~real Lie algebra, one can embed $\g$ inside the direct sum $\g\oplus \g$ by means of the map
\begin{gather*}
\g\rightarrow\g\oplus\g,
\qquad
z\mapsto\big(z,-z^{\dag}\big).
\end{gather*}
Under this identification, it is immediately verified that the elements in the statement of the
proposition get sent to a~basis of $\g_{\varepsilon}^{\R}$.
\end{proof}
\begin{Proposition}
Suppose $\varepsilon_r\neq 0$ for all $r$.
Then $\g_{\varepsilon}^{\nc}$ is a~semi-simple real Lie algebra.

\end{Proposition}
\begin{proof}
In this case, the projection onto the first coordinate of $\g\oplus \g$ provides an isomorphism between
$\g_{\varepsilon}$ and $\g$.
\end{proof}

Hence if $\varepsilon\in \Char_{\{-1,1\}}(Q^+)$, the Lie algebra $\g_{\varepsilon}^{\nc}$ is the real form
of $\g$ corresponding to the involution
\begin{gather*}
(X_r^+)^{*}=\varepsilon_{r}X_r^-,
\qquad
h_r^*=h_r.
\end{gather*}
Accordingly, $\g_{\varepsilon}^{\nc}$ is a~real Lie algebra of equal rank, meaning that we can take
a~Cartan decomposition $\g^{\nc} = \ttt\oplus \p$ with $\ttt$ compact and $i\h_\R \subseteq \ttt$.
Conversely, it is easy to see that any equal rank semi-simple real Lie algebra can be realized as some
$\g_{\varepsilon}^{\nc}$.
We recall that the equal rank semi-simple real Lie algebras are precisely those which admit an irreducible
unitary representation of discrete type.

As an example, let us present the simple equal rank real Lie algebras of type $(A)$ (see~\cite[Chapter
X]{Hel1}).
For $r\in \{0,\ldots,l\}$, choose $\sigma_r \in \{\pm 1\}$ such that $\varepsilon_r =
\sigma_{r-1}\sigma_{r}$ for $r\in \{1,\ldots,l\}$.
Then $\g_{\varepsilon}^{\nc}$ is always of type $(AIII)$, and $\varepsilon$ corresponds to $\su(p,l+1-p)$
with $p$ the number of negative $\sigma_r$.
This precise correspondence between the $\varepsilon$ and the various real Lie algebras is easy to
determine explicitly from the standard descriptions of the Cartan decompositions.

More generally, we obtain from Proposition~\ref{PropIden} the following characterization of which Lie
algebras appear as some $\g_{\varepsilon}$.
\begin{Corollary}
Let $\p=\llll\oplus \uu$ be a~parabolic Lie subalgebra of $\g$ with Levi factor $\llll$ and largest
nilpotent ideal $\uu$.
Let $\llll^{\nc}$ be a~real form of $\llll$ which is of equal rank on each simple summand and compact on
the center.
Then $\llll^{\nc}\oplus \uu\cong \g_{\varepsilon}^{\nc}$ as real Lie algebras for some $\varepsilon\in
\Char_{\R}(Q^+)$, and all~$\g_{\varepsilon}^{\nc}$ arise in this way.
\end{Corollary}

Our next step is to present $\g_\varepsilon$ by means of generators and relations.
\begin{Definition}
Fix $\varepsilon \in \Char_{\R}(Q^+)$.
We define $\widetilde{\g}_{\varepsilon}$ as a~universal Lie algebra in the following way.
A set of generators is given by a~triple of elements $X_r^{\pm}$ and $H_r$ for each simple root $\alpha_r$.
The relations can be grouped into four parts, which we label as `(H)-condition', `(T)orus action', `(S)erre
relations' and `(C)oupling conditions': for all $r,s \in I$,
\begin{enumerate}\itemsep=0pt
\item[(H)] $\lbrack H_r,H_s\rbrack = 0$,
\item[(T)] $\lbrack H_r,X_{s}^{\pm}\rbrack = \pm a_{rs}
X_{s}^{\pm}$,
\item[(S)] $\ad(X_r^{\pm})^{1-a_{rs}}(X_s^{\pm}) = 0$ when $r\neq s$,
\item[(C)$_\varepsilon$] $\lbrack X_r^+,X_s^-\rbrack = \delta_{rs}\varepsilon_r H_{r}$.
\end{enumerate}
\end{Definition}

It is immediate that $\widetilde{\g}_\varepsilon$ can be endowed with a~$^*$-operation such that $(X_r^+)^*
= X_r^-$ and $H_r^* = H_r$.
\begin{Proposition}
The Lie $^*$-algebras $\widetilde{\g}_\varepsilon$ and $\g_\varepsilon$ are isomorphic.
\end{Proposition}
\begin{proof}
It is straightforward to verify that there is a~unique Lie $^*$-algebra homomorphism
$\phi\colon \widetilde{\g}_{\varepsilon} \rightarrow \g_{\varepsilon}\subseteq \g\oplus\g$ such that $\phi(X_r^+)
= (\varepsilon_r X_r^+,X_r^+)$, $\phi(X_r^-) = (X_r^-,\varepsilon_rX_r^-)$ and $\phi(H_r) = (h_r,h_r)$.
It is obviously surjective.
On the other hand, it is easy to see by induction that each element of~$\widetilde{\g}_{\varepsilon}$ can
be written in the form $x+y+z$ with $x$ in the Lie algebra generated by the $X_r^+$'s, $z$ in the Lie
algebra generated by the $X_r^-$'s, and $y$ in the linear span of the $H_r$.
Since $\n^{\pm}$ are universal with respect to the relations~(S), we find that
$\dim(\widetilde{\g}_{\varepsilon}) \leq \dim(\g_{\varepsilon})$, hence $\phi$ is bijective.
\end{proof}

\subsection*{Acknowledgments}

It is a~pleasure to thank the following people for discussions on topics
related to the subject of this paper: J.~Bichon, P.~Bieliavsky, H.P.~Jakobsen, E.~Koelink, S.~Kolb,
U.~Kr\"{a}hmer and S.~Neshveyev.

\pdfbookmark[1]{References}{ref}
\LastPageEnding

\end{document}